\documentclass[a4paper,american,reqno]{amsart}

\newif\ifPreprint \Preprinttrue
\newif\ifSubmission \Submissionfalse

\usepackage{babel}
\usepackage[utf8]{inputenc}
\usepackage[T1]{fontenc}
\usepackage[binary-units=true]{siunitx}
\usepackage{booktabs}
\usepackage{eurosym}
\usepackage{algorithm,algorithmic}
\usepackage{xspace}
\usepackage{accents}
\usepackage{todonotes}
\usepackage{relsize}
\usepackage{tikz}
\usepackage{pgfplots}
\usepackage{csquotes}
\usepackage{microtype}
\usepackage[style = numeric-comp,
            maxbibnames = 100,
            maxcitenames = 2,
            giveninits = true,
            isbn = false,
            backend = biber]{biblatex}
\usepackage[colorlinks,
            citecolor=blue,
            urlcolor=blue,
            linkcolor=blue]{hyperref} 

\makeatletter
\patchcmd{\@settitle}{\uppercasenonmath\@title}{\scshape\large}{}{}
\patchcmd{\@setauthors}{\MakeUppercase}{\scshape\normalsize}{}{}
\makeatother

\vfuzz \hfuzz
\usetikzlibrary{external}
{dhn-nl-opt_preprint} 

\tikzexternaldisable

\newcommand{\abbr}[1][abbrev]{#1.\xspace}

\newcommand{\cf}{\abbr[cf]}
\newcommand{\eg}{\abbr[e.g]}

\newcommand{\ie}{\abbr[i.e]}

\newcommand{\wrt}{\abbr[w.r.t]}

\newtheorem{theorem}{Theorem}

\providecommand{\sfnamesize}{\relsize{0}}
\providecommand{\sfname}[1]{\textsf{\sfnamesize#1}\xspace}

\newcommand{\CONOPTfour}{\sfname{CONOPT4}}

\newcommand{\GAMS}{\sfname{GAMS}}

\newcommand{\Ipopt}{\sfname{Ipopt}}

\newcommand{\KNITRO}{\sfname{KNITRO}}

\newcommand{\SNOPT}{\sfname{SNOPT}}

\newcommand{\AROMA}{\sfname{AROMA}}
\newcommand{\STREET}{\sfname{STREET}}

\newcommand{\field}{\mathbb}

\newcommand{\complex}{\field{C}}

\newcommand{\naturals}{\field{N}}

\makeatletter
\@ifundefined{C}{\newcommand{\C}{\complex}}
{\PackageWarning{math-fields}{Use \string\field{C} for complex numbers;
 \MessageBreak command \string\C\space already defined}{}}
\makeatother
\newcommand{\N}{\naturals}

\newcommand{\graphset}{}

\newcommand{\Graph}{\graphset{G}}
\newcommand{\Arcs}{\graphset{A}}

\newcommand{\Nodes}{\graphset{V}}
\newcommand{\Vertices}{\Nodes}

\newcommand{\node}{u}
\newcommand{\otherNode}{v}

\newcommand{\arc}{a}

\newcommand{\Outedges}{\delta^{\text{out}}}

\newcommand{\Inedges}{\delta^{\text{in}}}

\newcommand{\Activeset}[1]{\mathcal{A}\ifx#1\empty\left(#1\right)\fi}

\newcommand{\activeset}[1]{\mathcal{A}\ifx#1\empty\else(#1)\fi}

\newcommand{\primalinfeasibility}[1][]{\theta_{\mathrm{pri}\ifx#1\empty\else{,{#1}}\fi}}

\newcommand{\red}{\reduced}
\newcommand{\kktrighthandside}[1][]{\omega\ifx#1\empty\else_{#1}\fi}
\newcommand{\kktrhs}{\kktrighthandside}
\newcommand{\kktreducedrighthandside}[1][\empty]{\kktrhs\ifx#1\empty_{\text{\red}}\else_{\text{\red,#1}}\fi}

\newcommand{\reduced}{\text{r}}

\newcommand{\steplength}[2][]{\alpha\ifx#1\empty_{#2}\else_{#2,#1}\fi}

\newcommand{\st}{\text{s.t.}}

\makeatletter
\newcommand{\fcdot}{\,\cdot\,}
\newcommand{\fcarg}[1]{\def\fc@rg{#1}\ifx\fc@rg\empty\fcdot\else\fc@rg\fi}
\makeatother
\newcommand{\abs}[1]{\lvert\fcarg{#1}\rvert}
\newcommand{\Abs}[1]{\left\lvert#1\right\rvert}

\newcommand{\defset}[3][\defsep]{\set{#2#1#3}}

\newcommand{\dparlong}[2]{\frac{\partial#2}{\partial#1}}

\newcommand{\dpar}{\dparlong}

\newcommand{\norm}[2][]{\lVert\fcarg{#2}\rVert\ifx#1\empty\else_{#1}\fi}
\newcommand{\Norm}[2][]{\left\lVert#2\right\rVert\ifx#1\empty\else_{#1}\fi}

\newcommand{\set}[1]{\{#1\}}

\newcommand{\snorm}[2][]{\lvert\!\lvert\!\lvert
  \fcarg{#1}\rvert\!\rvert\!\rvert\ifx#2\empty\else_{#1}\fi}
\newcommand{\Snorm}[2][]{\left\lvert\!\left\lvert\!\left\lvert
  #2\right\rvert\!\right\rvert\!\right\rvert\ifx#1\empty\else_{#1}\fi}
\newcommand{\sprod}[3][]{%
  \langle\fcarg{#2},\fcarg{#3}\rangle\ifx#1\empty\else_{#1}\fi}
\newcommand{\Sprod}[3][]{%
  \left\langle\fcarg{#2},\fcarg{#3}\right\rangle\ifx#1\empty\else_{#1}\fi}

\newcommand{\optmathindex}[1]{\ifx#1\empty\else,#1\fi}
\newcommand{\opttextindex}[1]{\ifx#1\empty\else,\text{#1}\fi}
\newcommand{\optmathsb}[1]{\ifx#1\empty\else_{#1}\fi}
\newcommand{\opttextsb}[1]{\ifx#1\empty\else_{\text{#1}}\fi}
\newcommand{\optmathsp}[1]{\ifx#1\empty\else^{#1}\fi}
\newcommand{\opttextsp}[1]{\ifx#1\empty\else^{\text{#1}}\fi}

\newcommand{\continuousFunctions}[1]{\mathcal{C}\ifx#1\empty\else^{#1}\fi}

\newcommand{\piecewiseContinuousFunctions}[1]{\mathcal{C}_p\ifx#1\empty\else^{#1}\fi}

\newcommand{\define}{\mathrel{{\mathop:}{=}}}

\newcommand{\diff}{\xspace\,\mathrm{d}}

\makeatletter
\newcommand{\objref}[4]{\def\obj@rg{#4}%
  #1\ifx\obj@rg\empty#2\else#3\xspace\ref{#4}--\fi\ref}
\makeatother
\newcommand{\Sobjref}[1]{\objref{#1}{~}{s}}

\newcommand{\Figref}[1][]{\Sobjref{Figure}{#1}}

\newcommand{\Secref}[1][]{\Sobjref{Section}{#1}}
\newcommand{\Tabref}[1][]{\Sobjref{Table}{#1}}

\newcommand{\area}{A}
\newcommand{\bore}{D}

\newcommand{\density}{\rho}
\newcommand{\dens}{\density}
\newcommand{\diameter}{\bore}
\newcommand{\diam}{\diameter}

\newcommand{\gravity}{g}
\newcommand{\grav}{\gravity}

\newcommand{\length}{L}
\newcommand{\massflow}{q}
\newcommand{\mflow}{\massflow}

\newcommand{\Power}{P}
\newcommand{\pressure}{p}
\newcommand{\press}{\pressure}

\newcommand{\roughness}{\rugosity}
\newcommand{\rugosity}{k}
\newcommand{\temperature}{T}
\newcommand{\temp}{\temperature}
\newcommand{\velocity}{v}
\newcommand{\vel}{\velocity}

\newcommand{\MS}[1]{}
\newcommand{\MSil}[1]{}
\newcommand{\RK}[1]{}
\newcommand{\RKil}[1]{}
\renewcommand{\MS}[1]{\todo[author=MS,color=orange!50,size=\small]{#1}}
\renewcommand{\MSil}[1]{\todo[inline,author=MS,color=orange!50,size=\small]{#1}}
\renewcommand{\RK}[1]{\todo[author=RK,color=green!50,size=\small]{#1}}
\renewcommand{\RKil}[1]{\todo[inline,author=RK,color=green!50,size=\small]{#1}}

\definecolor{luh-dark-blue}{rgb}{0.0, 0.313, 0.608}

\newcommand{\powerprofile}[1]{
  \begin{tikzpicture}
  \def \thickness {very thick}
  \begin{axis}[
  scale only axis,
  xmin=0, xmax=24,
  xlabel=$t$ (\si{h}),
  ylabel=Power (\si{\MW}),
  ylabel near ticks,
  grid=both,
  legend pos=outer north east]
  \addplot[blue, dashed, \thickness] table [x=t,
  y=$P_c$]{power-profiles/#1.dat};
  \addlegendentry{$\sum_{\arc \in \Ac} P_a(t)$}
  \addplot[brown, dashdotted, \thickness] table [x=t,
  y=$P_w_upper$]{power-profiles/#1.dat};
  \addlegendentry{$\Powerwaste^+$}
  \addplot[green, \thickness] table [x=t, y=$P_w$]{power-profiles/#1.dat};
  \addlegendentry{$\Powerwaste(t)$}
  \addplot[red, dotted, \thickness] table [x=t, y=$P_p$]{power-profiles/#1.dat};
  \addlegendentry{$P_p(t)$}
  \end{axis}
  \end{tikzpicture}
}

\newcommand{\powerprofilenobound}[1]{
  \begin{tikzpicture}
    \def \thickness {very thick}
    \begin{axis}[
    scale only axis,
    xmin=0, xmax=24,
    xlabel=$t$ (\si{h}),
    ylabel=Power (\si{\MW}),
    ylabel near ticks,
    grid=both,
    legend pos=outer north east]
    \addplot[blue, dashed, \thickness] table [x=t,
    y=$P_c$]{power-profiles/#1.dat};
    \addlegendentry{$\sum_{\arc \in \Ac} P_a(t)$}
    \addplot[green, \thickness] table [x=t, y=$P_w$]{power-profiles/#1.dat};
    \addlegendentry{$\Powerwaste(t)$}
    \addplot[red, dotted, \thickness] table [x=t,
    y=$P_p$]{power-profiles/#1.dat};
    \addlegendentry{$P_p(t)$}
    \end{axis}
  \end{tikzpicture}
}

\newcommand{\tempmflowprofile}[1]{
  \begin{tikzpicture}
    \def \thickness {very thick}
    \begin{axis}[
    scale only axis,
    xmin=0, xmax=24,
    ymin=345, ymax=410,
    axis y line*=left,
    xlabel=$t$ (\si{h}),
    ylabel=Temperature (\si{\K}),
    ylabel near ticks,
    grid=both,
    legend style={at={(1.07,1)},anchor=north west}]
    \addplot[red, \thickness] table [x=t, y=$T_a$]{power-profiles/#1.dat};
    \addlegendentry{$\temp_{\arc:\otherNode}(t)$}
    \end{axis}
    \begin{axis}[
    scale only axis,
    xmin=0, xmax=24,
    axis y line*=right,
    axis x line=none,
    ylabel=Mass flow (\si{\kg \per \s}),
    ylabel near ticks,
    legend style={at={(1.07,0.9)},anchor=north west}]
    \addplot[blue, \thickness] table [x=t, y=$q_a$]{power-profiles/#1.dat};
    \addlegendentry{$\mflow_{\arc}(t)$}
    \end{axis}
  \end{tikzpicture}
}

\newcommand{\Vdh}{\Nodes}
\newcommand{\Adh}{\Arcs}

\newcommand{\Vff}{\Nodes_{\text{ff}}}
\newcommand{\Vbf}{\Nodes_{\text{bf}}}
\newcommand{\Aff}{\Arcs_{\text{ff}}}
\newcommand{\Abf}{\Arcs_{\text{bf}}}
\newcommand{\Ac}{\Arcs_{\text{c}}}
\newcommand{\ad}{\arc_{\text{d}}}
\newcommand{\Apos}{\Arcs_{\text{pos}}}
\newcommand{\Aneg}{\Arcs_{\text{neg}}}

\newcommand{\inflowArcs}{\mathcal{I}}
\newcommand{\outflowArcs}{\mathcal{O}}

\newcommand{\nodeArcs}{\delta}

\newcommand{\costcoefficient}{\omega}
\newcommand{\costcoeff}{\costcoefficient}
\newcommand{\costcoeffwaste}{\costcoeff_{\mathrm{w}}}
\newcommand{\costcoeffgas}{\costcoeff_{\mathrm{g}}}
\newcommand{\costcoeffpress}{\costcoeff_{\mathrm{p}}}

\newcommand{\heattransfercoefficient}{U}
\newcommand{\heattrans}{\heattransfercoefficient}

\newcommand{\negativemassflow}{\gamma}
\newcommand{\negmflow}{\negativemassflow}
\newcommand{\positivemassflow}{\beta}
\newcommand{\posmflow}{\positivemassflow}
\newcommand{\Powerwaste}{\Power_{\mathrm{w}}}
\newcommand{\Powergas}{\Power_{\mathrm{g}}}
\newcommand{\Powerpress}{\Power_{\mathrm{p}}}
\newcommand{\slope}{h'}
\newcommand{\soiltemp}{\temp_0}
\newcommand{\specificheatcapacity}{c_{\mathrm{p}}}
\newcommand{\stagnation}{\mathrm{s}}
\newcommand{\heatcap}{\specificheatcapacity}
\newcommand{\spacediff}{\Delta x}
\newcommand{\spacediffarc}{\spacediff_\arc}
\newcommand{\temperaturebackflow}{\temp^{\text{bf}}}
\newcommand{\tempbf}{\temperaturebackflow}

\newcommand{\temperatureforeflowarc}{\temp^{\text{ff}}_\arc}

\newcommand{\tempffarc}{\temperatureforeflowarc}

\newcommand{\tempdiffvar}{\Delta \temp}
\newcommand{\tend}{T}
\newcommand{\Tint}{\mathcal{T}}
\newcommand{\timediff}{\Delta t}

\newcommand{\rev}[1]{#1}

\bibliography{dhn-nl-opt}

\begin{document}

\title{Nonlinear Optimization of District Heating Networks}
\author[R. Krug, V. Mehrmann, M. Schmidt]%
{Richard Krug, Volker Mehrmann, Martin Schmidt}

\address[R. Krug]{%
  Friedrich-Alexander-Universität Erlangen-Nürnberg,
  Discrete Optimization,
  Cauerstr.~11,
  91058~Erlangen,
  Germany}
\email{richard.krug@fau.de}

\address[V. Mehrmann]{%
  Institute for Mathematics, MA 4-5, TU Berlin,
  Straße des 17. Juni 136, 10623 Berlin, Germany}
\email{mehrmann@math.tu-berlin.de}

\address[M. Schmidt]{%
  Trier University,
  Department of Mathematics,
  Universitätsring 15,
  54296 Trier,
  Germany}
\email{martin.schmidt@uni-trier.de}

\date{\today}

\begin{abstract}
  We develop a complementarity-constrained nonlinear optimization
model for the time-dependent control of district heating networks.
The main physical aspects of water and heat flow in these networks are
governed by nonlinear and hyperbolic 1d partial differential equations.
In addition, a pooling-type mixing model is required at the nodes of the
network to treat the mixing of different water temperatures.
This mixing model can be recast using suitable complementarity constraints.
The resulting problem is a mathematical program with complementarity
constraints subject to nonlinear partial differential equations
describing the physics.
In order to obtain a tractable problem, we apply suitable
discretizations in space and time, resulting in a finite-dimensional
optimization problem with complementarity constraints for which we
develop a suitable reformulation with improved constraint regularity.
Moreover, we propose an instantaneous control approach for the
discretized problem, discuss practically relevant penalty formulations,
and present preprocessing techniques that are used to simplify the mixing
model at the nodes of the network.
Finally, we use all these techniques to solve realistic instances.
Our numerical results show the applicability of our techniques in
practice.


\end{abstract}

\keywords{District heating networks,
Nonlinear optimization,
Euler equations,
Differential-algebraic equations,
Mixing,
Complementarity constraints%
%
%
}
\subjclass[2010]{90-XX, 
90Cxx, 
90C30, 
90C35, 
90C90
%
}

\maketitle

\section{Introduction}
\label{sec:introduction}

Many countries in the world are striving to make a transition towards
an energy system that is mainly based on using energy from renewable
sources like wind and solar power, complemented by classical energy
sources like gas, oil, coal, or waste incineration.
The increasing use of highly fluctuating renewable energy sources
leads to many challenging problems from the engineering, mathematical,
and economic point of view.
A key to the success of this energy transition is the efficient and
intelligent coupling of the energy resources and the optimal operation
of the energy networks and energy storage.
In this direction, district heating networks play an important role,
since they can be used as energy storage, \eg, to balance fluctuations
at the electricity exchange.
To this end, district heating networks need to be operated
efficiently so that no unnecessary energy is used and, on the other
hand, security of supply should not be compromised.
This is a hard task since uncertainties of the heat
demand of households need to be considered and because the
physics-based time delays in these networks make it difficult to
react to changes in short periods of time.

To make the described intelligent use of district heating networks
possible, one needs (i) a proper mathematical model of the
network as well as fast and stable (ii) simulation and (iii)
optimization techniques.
In this paper, we develop a continuous optimization model for the
short-term optimal operation of a district heating network.
To this end, we assume that the heat demand of the households is
given and set up a nonlinear optimization model (NLP) for the control
of the heat supply and the pressure control of the network.
The building blocks of the entire model are nonlinear models of the
households, where thermal energy is withdrawn, the network
depot, in which the heat is supplied to the network and the pressure
is controlled, and a model of the transport network itself.

The model of the transport network is governed by two main
mathematical components; a system of one-dimensional (1d) nonlinear
hyperbolic partial differential equations (PDEs) to model the
relations of mass flows, water pressure, and temperature in a pipe
over time, and  a system of algebraic equations that is used at every
node of the network to model mass conservation, pressure continuity, and the
mixing of water temperatures.
The last aspect is very challenging, since these mixing models are
genuinely nonsmooth due to their dependence on flow directions, which
are part of the solution of the PDE and not known a priori.
To avoid integer-valued variables, we develop a mixing model using
complementarity constraints.
In summary, we consider a PDE-constrained nonlinear
mathematical program with complementarity constraints (MPCC),
which is a highly challenging class of optimization problems;
see, \eg, \cite{luo_pang_ralph_1996}.

Somehow surprisingly, there is not much literature about the mathematical
optimization of district heating networks.
A branch of applied publications focuses on specific case
studies. For instance, in~\cite{Pirouti_et_al:2013}, a case study for a
simplified model of a district heating project in South Wales is
carried out. The focus is more on an economic analysis than on mathematical and
physical modeling or optimization techniques.
The resulting problems are solved by a linear solver invoked in
a sequential linear programming approach.
A more general discussion about the technology and potentials of
district heating networks is presented in \cite{Rezaie_Rosen:2012}.
In \cite{Schweiger_et_al:2017}, the authors discuss different discrete
and continuous optimization problems.
As in our contribution, the authors start with a PDE-constrained
optimization problem and apply the first-discretize-then-optimize
approach yielding a finite-dimensional problem that is then solved.
Energy storage or storage tanks combined with district heating
networks are discussed in \cite{Colella_et_al:2012,Verda_Colella:2011}
and the impact of load variations and the integration of solar energy
is considered in \cite{Ben_Hassine_Eicker:2013}.
The design of district heating networks for stationary mathematical models is
carried out in \cite{Roland_Schmidt:2020,Bordin_et_al:2016,Dorfner_Hamacher:2014}.
In contrast to the mid- to long-term planning problems addressed in
these papers, in \cite{Sandou_et_al:2005}, the authors consider a
model predictive control (MPC) approach for computing a good
operational control of a network with a given design.
The resulting models are continuous nonlinear problems that need to be
solved in every iteration of the MPC loop.
A related approach is discussed in \cite{Verrilli_et_al:2017}, where an
MPC control is computed for a district heating system with thermal
energy storage and flexible loads.
Numerical simulation of district heating networks using a local time
stepping method is studied in \cite{Borsche_et_al:2018} and model
order reduction techniques for the hyperbolic equations in district
heating networks are discussed \rev{in \cite{Rein_Mohring_et_al:2019}
or~\cite{Rein_et_al:2018,Rein_et_al:2019}.
In the last two} papers, however, no optimization tasks are
considered.

As discussed above, a very important aspect of district heating
network models is the mixing of different water temperatures at the
nodes of the network.
Since the models are similar, related literature can also be
found in the field of optimization for gas transport networks; \cf,
\eg, \cite{van_der_Hoeven:2004,%
  Schmidt_et_al:2014,%
  Geissler_et_al:2018,%
  Schmidt_et_al:2016,%
  Geissler_et_al:2015,%
  Hante_Schmidt:2017}.

Our contribution is to consider the
optimization of district heating networks at a great level of detail and
physical accuracy; see Section~\ref{sec:modeling} for our modeling approach
that includes both 1d nonlinear PDEs and mixing models.
In order to obtain tractable optimization problems, we present tailored
discretizations of the PDEs in space and time in
Section~\ref{sec:discretized-distr-heat-network} and also provide
different equivalent formulations for the nodal mixing conditions; see
Section~\ref{sec:mixing-models}.
In Section~\ref{sec:optimization-techniques}, we present
problem-specific optimization techniques that enable us to solve
instances on realistic networks with reasonable space and time
discretizations. To be more specific, we set up an instantaneous
control approach that can both be used stand-alone and as a procedure
for computing initial values of good quality for the problem on the
entire time horizon.
Additionally, we derive suitable penalty formulations of the problem
that render the instances numerically more tractable.
Moreover, we present an easy-but-useful preprocessing technique to
decide flow directions in advance so that the amount of nonsmoothness
and the number of complementarity constraints for modeling the nodal mixing
conditions is reduced.
The described  techniques are then used to solve realistic instances in
Section~\ref{sec:numerical-results}.
Finally, we close the paper with a conclusion and some comments on
possible directions of future work in Section~\ref{sec:conclusion}.


\section{Modeling}
\label{sec:modeling}

We use a connected and directed graph $\Graph = (\Vertices, \Arcs)$ to model
the district heating network.
The network consists of
\begin{itemize}
\item a forward-flow part, which provides the consumers with hot water;
\item consumers, that use the hot water for heating;
\item a backward-flow part, which transports the cooled water back to the
  depot;
\item and the depot, where the heating of the cooled water takes
  place.
\end{itemize}
See Figure~\ref{fig:sample-network} for a schematic district
heating network.

The nodes $\Vdh = \Vff \cup \Vbf$
are the disjoint union of nodes~$\Vff$ of the forward-flow part and
nodes~$\Vbf$ of the backward-flow part of the network.
The arcs $\Adh$ are divided into forward-flow arcs~$\Aff$,
backward-flow arcs $\Abf$, consumer arcs $\Ac$, and the depot arc
$\ad$ of the district heating network provider.
Therefore, $\Arcs = \Aff \cup \Abf \cup \Ac \cup
\set{\ad}$ and we have
\begin{align*}
  \arc = (\node, \otherNode) \in \Aff & \implies \node \in \Vff, \
                                        \otherNode \in \Vff,\\
  \arc = (\node, \otherNode) \in \Abf & \implies \node \in \Vbf, \
                                        \otherNode \in \Vbf,\\
  \arc = (\node, \otherNode) \in \Ac & \implies \node \in \Vff, \
                                       \otherNode \in \Vbf,\\
  \ad = (\node, \otherNode) & \implies \node \in \Vbf, \
                              \otherNode \in \Vff.
\end{align*}
We optimize the district heating network in the time horizon $\Tint
\define [0, \tend]$ with predefined final time~$\tend > 0$.
In what follows, we introduce mathematical models for the different
parts of the network; namely pipes, nodes, consumers, and the depot of
the network provider.
After that, we introduce bounds for some of the quantities and state the
objective function.
To conclude this section, we summarize the parts to obtain a complete model of
the entire district heating network.
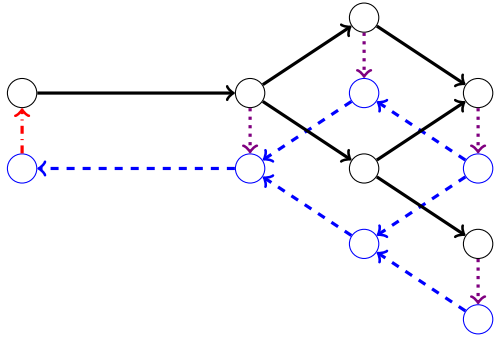
\begin{figure}
  \tikzexternalenable
  \tikzsetnextfilename{tikz-imgs/dhn}
  \begin{tikzpicture}
  \def \radius {11pt}
  \def \locations {3/0/1, 4.5/1/2, 4.5/-1/3, 6/0/4, 6/-2/5}
  \def \distance {1}
  \def \lineThickness {very thick}

  \node[circle, draw, minimum width=\radius](s) at (0,0){ };
  \foreach \x/\y/\name in \locations
  \node[circle, draw, minimum width=\radius](\name) at (\x,\y){ };

  \node[circle, draw, minimum width=\radius, blue](t)
  at (0,0 - \distance){};
  \foreach \x/\y/\name in \locations
  \node[circle, draw, minimum width=\radius, blue](R\name) at (\x,\y -
  \distance){ };

  \foreach \source/\dest in {R5/R3, R4/R3, R4/R2, R3/R1, R2/R1, R1/t}
  \draw [->, blue, dashed, \lineThickness] (\source) -- (\dest);

  \foreach \source/\dest in {s/1, 1/2, 1/3, 2/4, 3/4, 3/5}
  \draw [->, \lineThickness] (\source) -- (\dest);

  \foreach \n in {1,2,4,5}
  \draw [->, violet, dotted, \lineThickness] (\n) -- (R\n);

  \draw [->, red, dashdotted, \lineThickness] (t) -- (s);
\end{tikzpicture}
%
  \tikzexternaldisable

  \caption{A schematic district heating network: Forward-flow arcs are
    plotted in solid black, backward-flow arcs in dashed blue, consumers in
    dotted violet, and the depot in dashed-dotted red.}
  \label{fig:sample-network}
\end{figure}

\subsection{Pipe Modeling}
\label{sec:pipes}

We use the 1d Euler equations to model the physics of
hot water flow in the pipe network
\cite{Borsche_et_al:2018,Rein_et_al:2018,Koecher:2000}.
In what follows, we use $x \in [0, \length_\arc]$ to denote the spatial
coordinate, with  $\length_\arc$ being the length of pipe $\arc \in \Aff \cup \Abf$.
The continuity equation then is given by
\begin{equation}
  \label{eq:distr-heat-continuity}
  \dparlong{t}{\dens_\arc}(x,t) + \dparlong{x}{(\dens_\arc \vel_\arc)}(x,t) = 0,
  \quad \arc \in \Aff \cup \Abf.
\end{equation}
The 1d momentum equation for compressible
fluids in cylindrical pipes has the form
\begin{equation}
  \label{eq:distr-heat-euler-momentum}
  \begin{split}
    \dparlong{t}{(\dens_\arc \vel_\arc)}(x,t)
    + \dparlong{x}{\press_\arc}(x,t)
    + \dparlong{x}{(\dens_\arc \vel^2_\arc)}(x,t)
    \quad \\
    +\, \grav \dens_\arc(x,t) h'_\arc
    + \lambda_\arc \frac{\abs{\vel_\arc}
      \vel_\arc \dens_\arc}{2\diam_\arc}(x,t) & = 0,
    \quad \arc \in \Aff \cup \Abf;
  \end{split}
\end{equation}
see, \eg, \cite{Schmidt_et_al:2014,Mehrmann_Schmidt_Stolwijk:2017}.

Here and in what follows, $\dens_\arc$, $\press_\arc$, and
$\vel_\arc$ denote the density, pressure, and velocity of the water
in pipe $\arc$.
Furthermore, $\diam_\arc$ is the diameter and $\slope_\arc$ is the
slope of pipe~$\arc$, which we assume to be constant.
The gravitational acceleration is denoted by $g$.
The friction factor~$\lambda_\arc$ for turbulent flow is modeled by the
flow-independent law of Nikuradse (see, \eg,
\cite{Koch_et_al:ch02:2015}), \ie,
\begin{equation*}
  \lambda_\arc = \left( 2 \log_{10} \left(
      \frac{\diam_\arc}{\roughness_\arc} \right) + 1.138 \right)^{-2},
  \quad \arc \in \Aff \cup \Abf,
\end{equation*}
where $\roughness_\arc$ is the roughness of the inner pipe wall.
We are aware that there are also other empirical models of the
friction factor for the turbulent case, which might also render
$\lambda$ being dependent on $x$ and $t$.
Moreover, there is Hagen--Poiseuille's exact law for laminar flow;
see, \eg, \cite{Koch_et_al:ch02:2015} and the references therein.
For the ease of presentation, we restrict ourselves to the law of
Nikuradse, which only depends on the data of the pipe.
However, other models can in principle also be incorporated.
For a list of all parameters and variables of the model see
\Tabref{tab:dhn-notation}\rev{, where we also distinguish between
directly controllable variables at the depot and physical state
variables in the network.}
\begin{table}
  \centering
  \caption{\rev{Controllable variables at the depot (top), physical
      state variables in the network (mid), and
      given parameters (bottom) of the district heating network
      model}}
  \label{tab:dhn-notation}
  \begin{tabular}{lll}
  \toprule
  Symbol & Explanation & Unit\\
  \midrule
  $\Powerwaste(t)$ & Power production through waste incineration & \si{\W}\\
  $\Powergas(t)$ & Power production through gas combustion & \si{\W}\\
  $\Powerpress(t)$ & Pumping power to increase the water pressure & \si{\W}\\
  \midrule
  $\dens_\arc(x,t)$ & Density of the water in pipe $\arc$ & \si{\kg \per
    \cubic \m}\\
  $\vel_\arc(x,t)$ & Flow velocity in pipe $\arc$ & \si{\m \per \s}\\
  $\press_\arc(x,t)$ & Pressure in pipe $\arc$ & \si{\Pa}\\
  $\temp_\arc(x,t)$ & Water temperature in pipe $\arc$ & \si{\K}\\
  $\mflow_\arc(x,t)$ & Mass flow in pipe $\arc$; $\mflow_\arc=\area_\arc
  \dens_\arc \vel_\arc$ & \si{\kg \per \s}\\
  $\press_\node(t)$ & Pressure at node $\node$ & \si{\Pa}\\
  $\temp_\node(t)$ & (Mixed) water temperature at node $\node$ & \si{\K}\\
  \midrule
  $t$ & Time coordinate; $t \in \Tint$ & \si{\s}\\
  $\Tint$ & Time horizon $\Tint \define [0,\tend]$ & ---\\
  $x$ & Spatial coordinate in a pipe & \si{\m}\\
  $\length_\arc$ & Length of pipe $\arc$ & \si{\m}\\
  $\diam_\arc$ & Diameter of pipe $\arc$ & \si{\m}\\
  $\area_\arc$ & Cross-sectional area of pipe $\arc$; $\area_\arc =
                 \pi\left(\diam_\arc / 2\right)^2$ & \si{\square \m}\\
  $\slope_\arc$ & Slope of pipe $\arc$ & 1\\
  $\lambda_\arc$ & Friction factor of pipe $\arc$ & 1\\
  $\Power_\arc(t)$ & Power consumption of the consumer at arc $\arc$ & \si{\W}\\
  $\roughness_\arc$ & Roughness of the inner wall of pipe $\arc$ & \si{\m}\\
  $\heattrans_\arc$ & Heat transfer coefficient of the wall of pipe $\arc$
  & \si{\W \per \square \m \per \K}\\
  $\tempffarc$ & Consumers' minimum inlet water temperature & \si{\K} \\
  $\tempbf$ & Consumers' outlet water temperature & \si{\K}\\
  $\soiltemp$ & Surrounding temperature & \si{\K}\\
  $\heatcap$ & Specific heat capacity of water & \si{\J \per \kg \per \K}\\
  $\press_\stagnation$ & Stagnation pressure of the network & \si{\Pa} \\
  $\xi_\Power$ & Max. change in power over time at depot & \si{\W \per \second} \\
  $\xi_\temp$ & Max. change in outlet temperature over time at depot & \si{\K \per \second} \\
  $\grav$ & Gravitational acceleration & \si{\m \per \square \s}\\
  $\costcoeffwaste$ & Cost coefficient for waste incineration & \si{\text{\euro} \per \watt} \\
  $\costcoeffgas$ & Cost coefficient for gas combustion & \si{\text{\euro} \per \watt} \\
  $\costcoeffpress$ & Cost coefficient for pumps & \si{\text{\euro} \per \Pa} \\
  \bottomrule
\end{tabular}
%

\end{table}

Since we assume that the water is incompressible, \ie,
\begin{equation}
  \label{eq:distr-heat-incompressibility}
  \dparlong{t}{\dens_\arc}(x,t) + \vel_\arc(x,t) \dparlong{x}{\dens_\arc}(x,t)
  = 0,\quad \arc
  \in \Aff \cup \Abf;
\end{equation}
see, \eg, \cite{Marsden_Chorin:1993} for details on fluid
flow modeling, we can rewrite the continuity equation
\eqref{eq:distr-heat-continuity} as
\begin{align*}
  0 & = \dparlong{t}{\dens_\arc}(x,t) + \dparlong{x}{(\dens_\arc
  \vel_\arc)}(x,t)\\
  & =  \dparlong{t}{\dens_\arc}(x,t)
  + \dens_\arc \dparlong{x}{\vel_\arc}(x,t)
  + \vel_\arc \dparlong{x}{\dens_\arc}(x,t)\\
  & = \dens_\arc \dparlong{x}{\vel_\arc}(x,t).
\end{align*}
Since the density $\dens_\arc(x, t)$ is always positive, we can divide by it
and obtain
\begin{equation*}
  \dparlong{x}{\vel_\arc}(x,t) = 0.
\end{equation*}
\rev{Using these consequences of incompressibility,
the momentum equation~\eqref{eq:distr-heat-euler-momentum} simplifies
to}
\begin{align*}
  & \quad \dparlong{t}{(\dens_\arc \vel_\arc)}(x,t)
  + \dparlong{x}{(\dens_\arc \vel^2_\arc)}(x,t)\\
  = \ & \dens_\arc \dparlong{t}{\vel_\arc}(x,t)
  + \vel_\arc \dparlong{t}{\dens_\arc}(x,t)
  + (\dens_\arc \vel_\arc) \dparlong{x}{\vel_\arc}(x,t)
  + \vel_\arc \dparlong{x}{(\dens_\arc \vel_\arc)}(x,t)\\
  = \ & \dens_\arc \dparlong{t}{\vel_\arc}(x,t)
  + \vel_\arc \left(\dparlong{t}{\dens_\arc}(x,t)
  + \dparlong{x}{(\dens_\arc \vel_\arc)}(x,t)\right)\\
  = \ & \dens_\arc \dparlong{t}{\vel_\arc}(x,t)
\end{align*}
and \rev{we thus} obtain the simplified 1d system of incompressible
Euler equations
\begin{subequations}
  \label{eq:distr-heat-euler-incompressible}
  \begin{align}
    \label{eq:distr-heat-continuity-incompressible}
    \dparlong{x}{\vel_\arc}(x,t) & = 0, \quad \arc \in \Aff \cup \Abf,
    \\
    \label{eq:distr-heat-euler-momentum-incompressible}
    \dens_\arc(x,t) \dparlong{t}{\vel_\arc}(x,t) +
    \dparlong{x}{\press_\arc}(x,t)
    + \grav \dens_\arc(x,t) h'_\arc
    \quad \nonumber \\
    +\, \lambda_\arc \frac{\abs{\vel_\arc}
    \vel_\arc \dens_\arc}{2\diam_\arc}(x,t)
                                 & = 0, \quad \arc \in \Aff \cup \Abf,
  \end{align}
\end{subequations}
that we use for setting up our optimization problem.

It should be noted that \eqref{eq:distr-heat-continuity-incompressible} implies constant
velocity in the pipe, \ie, $\vel_{\arc}(x,t) = \vel_{\arc}(t)$ for all $x \in
[0, \length_\arc]$.

The thermal energy equation for each pipe $\arc \in \Aff \cup \Abf$ is given by
\begin{equation}
  \label{eq:distr-heat-energy}
  \dparlong{t}{\temp_\arc}(x,t)
  + \vel_\arc(t) \dparlong{x}{\temp_\arc}(x,t)
  + \frac{4 \heattrans_\arc}{\heatcap \dens_\arc(x,t)
    \diam_\arc}(\temp_\arc(x,t)-\soiltemp)
  = 0,
  \quad \arc \in \Aff \cup \Abf;
\end{equation}
see \cite{Sandou_et_al:2005,Borsche_et_al:2018,Rein_et_al:2018}.
In \eqref{eq:distr-heat-energy}, $\temp_\arc$ describes the water temperature,
$\heattrans_\arc$ is the heat transfer
coefficient of the pipe's wall, $\heatcap$ is the specific heat capacity of
water, and $\soiltemp$ is the temperature in the environment
surrounding the pipe.

To close the system, one finally needs initial and boundary conditions
as well as an equation of state.
In the literature one can find formulas for the density of water
depending on the temperature; see, \eg, \cite{Koecher:2000}.
Since we make the incompressibility assumption
\eqref{eq:distr-heat-incompressibility}, in the context of our
optimization model, we assume as another simplification that the
density of the water is constant, \ie, $\dens_\arc(x, t) = \dens$.

This assumption allows us to rewrite the momentum equation
\eqref{eq:distr-heat-euler-momentum-incompressible} as follows:
\begin{equation*}
  \dparlong{x}{\press_\arc}(x,t)
  = - \dens \dparlong{t}{\vel_\arc}(t)
  - \grav \dens h'_\arc
  - \lambda_\arc \frac{\abs{\vel_\arc}
    \vel_\arc \dens}{2\diam_\arc}(t),
  \quad \arc \in \Aff \cup \Abf.
\end{equation*}
Since the right-hand side does not depend on the spatial
coordinate~$x$, the pressure $\press_\arc(x, t)$ is linear in $x$.
Thus, it holds \rev{that}
\begin{equation}
  \label{eq:distr-heat-euler-momentum-const-density}
  \frac{\press_\arc(\length_\arc,t) - \press_\arc(0,t)}{\length_\arc}
  = - \dens \dparlong{t}{\vel_\arc}(t)
  - \grav \dens h'_\arc
  - \lambda_\arc \frac{\abs{\vel_\arc}
    \vel_\arc \dens}{2\diam_\arc}(t),
  \quad \arc \in \Aff \cup \Abf.
\end{equation}

In this subsection, we have presented a simplified model of the 1d
compressible Euler equations for the description of the pipe
flow. More sophisticated models, or even complete hierarchies of
models for example those constructed in gas flow \cite{DomHLT17},
should be used for detailed simulation methods or the analysis of the
flow. However, in the context of our optimization methods, already the
discussed modeling level presents a mathematical and computational
challenge.

\subsection{Nodal Coupling Equations}
\label{sec:nodal-coupling-equations}

In this subsection, we expand our network model by suitable coupling
conditions on the nodes for mass flow, pressure, and
temperature. These conditions are modeled by algebraic equations. 

The mass balance equation for each node~$\node \in \Vdh$ is
described by
\begin{equation}
  \label{eq:distr-heat-mass-balance}
  \sum_{\arc \in \Inedges(\node)} \mflow_\arc(t) =
  \sum_{\arc \in \Outedges(\node)} \mflow_\arc(t),
  \quad \node \in \Vdh,\, t \in \Tint,
\end{equation}
where $\mflow_\arc = \area_\arc \dens \vel_\arc$ denotes the mass flow of
pipe $\arc$ with cross-sectional area $\area_\arc=\pi(\diam_\arc /
2)^2$.
Here and in what follows, we use the standard $\delta$-notation, \ie,
we define
\begin{align*}
  \Outedges(\node) & \define \defset{\arc \in \Arcs}{\exists
                     \otherNode \text{ with } \arc = (\node,
    \otherNode)},
  \\
  \Inedges(\node) & \define \defset{\arc \in \Arcs}{\exists \otherNode
                    \text{ with }\arc = (\otherNode, \node)},
\end{align*}
and $\delta(\node) \define \Outedges(\node) \cup \Inedges(\node)$.
Note that \eqref{eq:distr-heat-mass-balance} implies that we have no in-
and outflow to or from the network.

The pressure continuity equations for each node are given by
\begin{subequations}
  \label{eq:distr-heat-pressure-continuity}
  \begin{align}
    \label{eq:distr-heat-pressure-continuity-out}
    \press_\node(t) & = \press_\arc(0, t),
                      \quad \node \in \Vdh, \, \arc \in \Outedges(\node),\, t \in \Tint,\\
    \label{eq:distr-heat-pressure-continuity-in}
    \press_\node(t) & = \press_\arc(\length_\arc, t),
                      \quad \node \in \Vdh, \, \arc \in \Inedges(\node),\, t \in \Tint,
  \end{align}
\end{subequations}
where $\press_\node(t)$ denotes the pressure at node~$\node$;
see Figure~\ref{fig:pressure-continuity} for an illustration.
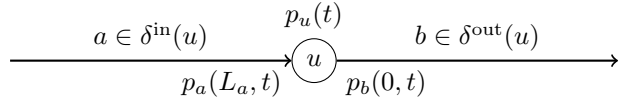
\begin{figure}
  \tikzexternalenable
  \tikzsetnextfilename{tikz-imgs/pressure-continuity}
  \begin{tikzpicture}
  \def \radius {16pt}
  \def \distance {4}
  \def \lineThickness {thick}

  \node[circle, draw, minimum width=\radius,label={$\press_\node(t)$}](u) at
  (0,0){$\node$};
  \draw [->, \lineThickness] (-\distance,0) -- (u)
  node[midway, above]{$a \in \Inedges(\node)$} node[at end, below
  left]{$\press_a(\length_a,t)$};
  \draw [->, \lineThickness] (u) -- (\distance,0)
  node[midway, above]{$b \in \Outedges(\node)$} node[at start, below
  right]{$\press_b(0,t)$};
\end{tikzpicture}
%
  \tikzexternaldisable

  \caption{Pressure continuity at node $u$}
  \label{fig:pressure-continuity}
\end{figure}
We also need to introduce temperature mixing equations to describe the
behavior of the water temperature in the nodes, where water of different temperatures is mixed.
Since the mixing model depends on the flow directions, we define the inflow and
outflow arcs of a node~$\node$ at a given time~$t \in \Tint$ as
\begin{align*}
  \inflowArcs(\node, t)
  & \define
    \defset{\arc \in \Inedges(\node)}{\vel_\arc(t) \geq 0}
    \cup
    \defset{\arc \in \Outedges(\node)}{\vel_\arc(t) \leq 0}, & & \node \in
                                                                 \Vdh,\,
                                                                 t \in \Tint,
  \\
  \outflowArcs(\node, t)
  & \define
    \defset{\arc \in \Inedges(\node)}{\vel_\arc(t) < 0}
    \cup
    \defset{\arc \in \Outedges(\node)}{\vel_\arc(t) > 0}, & & \node \in
                                                              \Vdh,\, t \in \Tint.
\end{align*}
The temperature mixing equations for each node are modeled as
\begin{subequations}
  \label{eq:distr-heat-temperature-mixing-with-cp}
  \begin{align}
    \label{eq:distr-heat-temperature-mixing-in-with-cp}
    \temp_\node(t) & = \frac{\sum_{\arc \in \inflowArcs(\node, t)}
                     \abs{\mflow_\arc(t)} \heatcap \temp_{\arc:\node}(t)}
                     {\sum_{\arc \in \inflowArcs(\node, t)} \abs{\mflow_\arc(t)}
                     \heatcap},
    & & \node \in \Vdh,\, t \in \Tint,
    \\
    \label{eq:distr-heat-temperature-mixing-out-with-cp}
    \temp_\node(t) & = \temp_{\arc:\node}(t),
    & & \node \in \Vdh, \ \arc \in \outflowArcs(\node, t),\, t \in \Tint,
  \end{align}
\end{subequations}
where $\temp_\node(t)$ denotes the mixed water temperature at
node~$\node$ and where we use the notation
\begin{equation*}
  \temp_{\arc:\node}(t) \define
  \begin{cases}
    \temp_\arc(0, t), & \node \in \Vdh, \ \arc \in \Outedges(\node),\, t \in
    \Tint,\\
    \temp_\arc(\length_\arc, t), & \node \in \Vdh, \ \arc \in
    \Inedges(\node),\, t \in \Tint;
  \end{cases}
\end{equation*}
see, \eg,
\cite{Schmidt_et_al:2014,Schmidt_et_al:2016,Hante_Schmidt:2017}, where
a similar model is considered for mixing effects in natural
gas transport networks.

Equation~\eqref{eq:distr-heat-temperature-mixing-in-with-cp}
can be derived from the conservation of
energy if the specific heat capacities in
\eqref{eq:distr-heat-temperature-mixing-with-cp} \rev{are independent
  of the water temperature.
  Since we consider the mixing of water only, the additional
  assumption that all heat capacities are the
  same is appropriate.
  Using this, \eqref{eq:distr-heat-temperature-mixing-with-cp} can be
  simplified to}
\begin{subequations}
  \label{eq:distr-heat-temperature-mixing}
  \begin{align}
    \label{eq:distr-heat-temperature-mixing-in}
    \temp_\node(t) & = \frac{\sum_{\arc \in \inflowArcs(\node, t)}
                     \abs{\mflow_\arc(t)} \temp_{\arc:\node}(t)}
                     {\sum_{\arc \in \inflowArcs(\node, t)}
                     \abs{\mflow_\arc(t)}},
    & & \node \in \Vdh,\, t \in \Tint,
    \\
    \label{eq:distr-heat-temperature-mixing-out}
    \temp_\node(t) & = \temp_{\arc:\node}(t),
    & & \node \in \Vdh, \ \arc \in \outflowArcs(\node, t),\, t \in \Tint.
  \end{align}
\end{subequations}
Obviously, the discussed mixing model is only defined at
nodes~$\node$ with inflow, \ie, if
\begin{equation*}
  \sum_{\arc \in \inflowArcs(\node, t)} \abs{\mflow_\arc(t)} > 0.
\end{equation*}

Note further that the mixing model in
\eqref{eq:distr-heat-temperature-mixing} cannot be used directly in an
optimization context because the sets~$\inflowArcs(\node, t)$ and
$\outflowArcs(\node, t)$ depend on the solution and are thus not known
a priori.
In Sections~\ref{sec:an-mpcc-based} and~\ref{sec:nlp-mixing-model}, we
present a reformulation of the mixing model that deals with this
difficulty.

\subsection{Consumer and Depot Models}
\label{sec:consumers-depot}

Consumers at arcs~$\arc = (\node, \otherNode) \in \Ac$ are
modeled by
\begin{subequations}
  \label{eq:distr-heat-consumer}
  \begin{align}
    \label{eq:distr-heat-consumer-mass-flow}
    \vel_\arc(t) & \geq 0, & & t \in \Tint,\\
    \label{eq:distr-heat-consumer-power-usage}
    \Power_\arc(t)
    & = \mflow_\arc(t) \heatcap \left(\temp_{\arc:\node}(t) -
    \temp_{\arc:\otherNode}(t)\right),
    & & t \in \Tint,\\
    \label{eq:distr-heat-consumer-temperature-end}
    \temp_{\arc:\otherNode}(t) & = \tempbf, & & t \in \Tint,\\
    \label{eq:distr-heat-consumer-temperature-start}
    \temp_{\arc:\node}(t) & \geq \tempffarc, & & t \in \Tint,\\
    \pressure_{\arc:\otherNode}(t) & \leq \pressure_{\arc:\node}(t), & &  t \in \Tint,
  \end{align}
\end{subequations}
where $\Power_\arc(t)$ is the given power consumption of the consumer
$\arc \in \Ac$, $\tempbf$ is the contractually agreed temperature
of the water that flows into the backward-flow network, and $\tempffarc$ is
the minimum inlet water temperature of the consumer~$\arc \in \Ac$.
Later in our numerical experiments, we will relax the equality
constraint~\eqref{eq:distr-heat-consumer-temperature-end} to
$\temp_{\arc:\otherNode}(t) \in [\tempbf - \varepsilon,
\tempbf + \varepsilon]$ for a small $\varepsilon > 0$\rev{,
since this leads to a significantly improved convergence behavior of the
tested solvers in our numerical experiments}.

The depot at arc~$\arc = \ad = (\node,\otherNode)$ is modeled by
\begin{subequations}
  \label{eq:distr-heat-depot}
  \begin{align}
    \label{eq:distr-heat-depot-mass-flow}
    \vel_{\arc}(t) & \geq 0, & & t \in \Tint,\\
    \label{eq:distr-heat-depot-stagnation-pressure}
    \press_\node(t) & = \press_\stagnation, & & t \in \Tint,\\
    \label{eq:distr-heat-depot-pressure}
    \Powerpress(t) & =
    \frac{\mflow_\arc(t)}{\dens}
    \left(\press_{\arc:\otherNode}(t) - \press_{\arc:\node}(t)\right),
    & & t \in \Tint,\\
    \label{eq:distr-heat-depot-power-production}
    \Powerwaste(t) + \Powergas(t)
    & =
    \mflow_\arc(t) \heatcap \left(\temp_{\arc:\otherNode}(t) -
    \temp_{\arc:\node}(t)\right),
    & & t \in \Tint,\\
    \Abs{\dpar{t}{\Powerwaste}(t)}
    & \leq \xi_\Power,
    & & t \in \Tint,\label{eq:distr-heat-depot-P-gradient}\\
    \Abs{\dpar{t}{\temp_{\arc:\otherNode}}(t)}
    & \leq \xi_\temp,
    & & t \in \Tint,\label{eq:distr-heat-depot-T-gradient}
  \end{align}
\end{subequations}
where $\press_\stagnation$ is the so-called stagnation pressure of the
network.
\rev{Since all other physical and technical equations of the model are
stated in pressure differences, the fixation of one pressure value
leads to unique pressure values everywhere in the network,
which is the reason for introducing the stagnation pressure.}
In our implementation, we however will allow a variation in an
interval $\press_\node(t) \in [\press_\stagnation - \varepsilon,
\press_\stagnation + \varepsilon]$ instead\rev{; \cf the relaxation of
  the backward-flow temperature
  constraint~\eqref{eq:distr-heat-consumer-temperature-end} above}.
The power to run the pumps to realize a pressure increase in the depot of the
district heating network provider is denoted by $\Powerpress(t)$.
A temperature gain is obtained by thermal power production in the depot.
The corresponding equation \eqref{eq:distr-heat-depot-power-production} is
similar to the power consumption
equation~\eqref{eq:distr-heat-consumer-power-usage} for consumers,
where $\Powerwaste(t)$ and $\Powergas(t)$ describe the thermal power produced
by waste incineration and gas combustion, respectively.
Finally, \eqref{eq:distr-heat-depot-P-gradient} and
\eqref{eq:distr-heat-depot-T-gradient} bound the change over time of
the power from waste incineration as well as the change over time of
the depot's outflow temperature.

\subsection{Bounds, Objective Function, and Model Summary}
\label{sec:bounds-objective-model}

The different variables of the network that are used in the model are subject to the following bounds for all $t \in \Tint$,
\begin{subequations}
  \label{eq:distr-heat-bounds}
  \begin{align}
    \press_\node(t) & \in [\press_\node^-,\press_\node^+],
    \quad
    \temp_\node(t) \in [\temp_\node^-,\temp_\node^+],
    \quad \node \in\Vdh,
    \\
    \Powerwaste(t) & \in [0,\Powerwaste^+],
    \quad
    \Powergas(t) \in [0,\Powergas^+],
    \quad
    \Powerpress(t) \in [0,\Powerpress^+].
  \end{align}
\end{subequations}
The objective function to minimize is given by
\begin{equation}
  \label{eq:distr-heat-cost-function}
  \int_0^{\tend} \left(\costcoeffwaste \Powerwaste(\tau)
    + \costcoeffgas \Powergas(\tau) + \costcoeffpress \Powerpress(\tau) \right)
    \diff \tau,
\end{equation}
where $\costcoeffwaste$, $\costcoeffgas$, and $\costcoeffpress$ are
cost coefficients of the waste incineration, the gas combustion, and the
pumping power, respectively.
Here, we assume that these cost coefficients are constant over time.
However, time-dependent costs can also be considered in a similar manner.
Note that, in principle, other methods of thermal power production,
\eg, power-to-heat, can be modeled in an analogous way.

In summary, we obtain the following nonlinear optimization problem with PDE
constraints
\begin{equation}
  \label{eq:distr-heat-opt-problem}
  \begin{split}
    \min \quad & \eqref{eq:distr-heat-cost-function}\\
    \st \quad & \text{Incompressible Euler equation: }
    \eqref{eq:distr-heat-euler-momentum-const-density},\\
    & \text{Thermal energy equation: }
    \eqref{eq:distr-heat-energy},\\
    & \text{Mass balance equation: }
    \eqref{eq:distr-heat-mass-balance},\\
    & \text{Pressure continuity equations: }
    \eqref{eq:distr-heat-pressure-continuity},\\
    & \text{Temperature mixing equations: }
    \eqref{eq:distr-heat-temperature-mixing},\\
    & \text{Consumer constraints: }
    \eqref{eq:distr-heat-consumer},\\
    & \text{Depot constraints: }
    \eqref{eq:distr-heat-depot},\\
    & \text{Bounds: }
    \eqref{eq:distr-heat-bounds}.
  \end{split}
\end{equation}
Note that \eqref{eq:distr-heat-opt-problem} is a nonsmooth and
infinite-dimensional nonlinear optimization problem subject to PDEs
and algebraic constraints.
\rev{While the separate parts of the model such as the incompressible Euler
equations or the mixing models at nodes are known in the literature,
the novelty of the modeling discussed here is the combination of these
aspects that leads to a highly accurate representation of the physical
behavior.}

Since we want to solve \rev{the presented model} as an NLP, we apply a
first-discretize-then-optimize approach by using suitable finite
difference discretizations \rev{of} the differential
equations. This will be discussed in the next section.


\section{PDE Discretizations}
\label{sec:discretized-distr-heat-network}

In this section, we discuss the discretization in space and time via
finite difference schemes.

\subsection{Implicit Euler Discretization in Space and Time}
\label{sec:impl-euler-discr}

For the time discretization,  we partition the time horizon~$\Tint =
[0, \tend]$ equidistantly in~$N+1 \in \N$ time points
\begin{equation*}
  t_i \define \frac{i \tend}{N},
  \quad i \in \set{0,\dots,N}.
\end{equation*}
Thus, the length of the discretization intervals is $\timediff
\define \tend / N$.

For the discretization in space of pipe $\arc \in \Aff \cup \Abf$, we
use $M_\arc +1 \in \N$ discretization points
\begin{equation*}
  x_{\arc,k} \define \frac{k \length_\arc}{M_\arc},\quad
  k \in \set{0,\dots,M_\arc},
  \quad \text{and} \quad
  \spacediffarc \define \frac{\length_\arc}{M_\arc}.
\end{equation*}
To obtain a large stability region for the method, we use an implicit
Euler discretization for the momentum equation
\eqref{eq:distr-heat-euler-momentum-const-density}, which leads to the
difference equation
\begin{equation}
  \label{eq:distr-heat-euler-const-density-discr}
  \begin{split}
    \dens \frac{\vel_\arc(t_{i+1}) -
      \vel_\arc(t_{i})}{\timediff}
    + \frac{\press_\arc(\length_\arc, t_{i+1}) - \press_\arc(0,
      t_{i+1})}{\length_\arc}
    \quad \\
    +\, \grav \dens \slope_\arc
    + \lambda_\arc \frac{\abs{\vel_\arc(t_{i+1})}
      \vel_\arc(t_{i+1}) \dens}{2\diam_\arc}
    & = 0
  \end{split}
\end{equation}
for $\arc \in \Aff \cup \Abf$ and $i \in \set{0,\dots,N - 1}$.
Note that in the context of a forward simulation, to avoid the
solution of \rev{(large)} nonlinear systems, we could have also used
an explicit integration scheme for the momentum equation.
However, since we are using the discretization method within an
optimization model, the implicit discretization does not lead to
increased costs anyway.

For the spatial semi-discretization of the thermal energy equation
\eqref{eq:distr-heat-energy} we use an implicit Euler discretization,
yielding
\begin{equation*}
  \label{eq:distr-heat-energy-semi-discr}
  \begin{split}
    \dparlong{t}{\temp_\arc}(x_{\arc, k + 1},t)
    + \vel_\arc(t) \frac{\temp_\arc(x_{\arc, k + 1}, t) - \temp_\arc(x_{\arc,
        k}, t)}{\spacediffarc}
    \quad \\
    +\, \frac{4 \heattrans_\arc}{\heatcap \dens
      \diam_\arc}(\temp_\arc(x_{\arc, k + 1},t)-\soiltemp)
    & = 0
  \end{split}
\end{equation*}
for $\arc \in \Aff \cup \Abf$ and $k \in \set{0,\dots,M_\arc - 1}$.
Note that in the optimality conditions for the discretized
optimization problem, which form a boundary value problem, there is
no preferred space direction, so we will discuss an
alternative approach based on central differences in the next
section.

The time discretization of the space-discretized
thermal energy equation is again done in an implicit way via
\begin{equation}
  \label{eq:distr-heat-energy-discr}
  \begin{split}
    \frac{\temp_\arc(x_{\arc, k + 1},t_{i+1}) - \temp_\arc(x_{\arc, k +
        1},t_{i})}{\timediff}
    \quad \\
    +\, \vel_\arc(t_{i+1}) \frac{\temp_\arc(x_{\arc, k + 1}, t_{i+1}) -
      \temp_\arc(x_{\arc, k}, t_{i+1})}{\spacediffarc}
    \quad \\
    +\, \frac{4 \heattrans_\arc}{\heatcap \dens
      \diam_\arc}(\temp_\arc(x_{\arc, k + 1},t_{i+1})-\soiltemp)
    & = 0
  \end{split}
\end{equation}
for $\arc \in \Aff \cup \Abf$, $k \in \set{0,\dots,M_\arc - 1}$, and $i \in
\set{0,\dots,N - 1}$.
The differential depot
constraints~\eqref{eq:distr-heat-depot-P-gradient}
and~\eqref{eq:distr-heat-depot-T-gradient} are discretized as
\begin{equation*}
  \frac{\abs{\Powerwaste(t_{i+1}) - \Powerwaste(t_{i})}}{\timediff}
  \leq \xi_\Power,
  \quad
  \frac{\abs{\temp_{\arc:\otherNode}(t_{i+1}) - \temp_{\arc:\otherNode}(t_i)}}{\timediff}
  \leq \xi_\temp,
  \quad
  i = 0, \dotsc, N-1.
\end{equation*}

Discretizing the algebraic equations just means formulating them for each
discretization point in time. For example, the discretized version of the mass
balance equation \eqref{eq:distr-heat-mass-balance} reads
\begin{equation*}
  \label{eq:distr-heat-mass-balance-discr}
  \sum_{\arc \in \Inedges(\node)} \mflow_\arc(t_i) = \sum_{\arc
    \in \Outedges(\node)}
  \mflow_\arc(t_i),\quad \node \in \Vdh,\ i \in \set{0,\dots,N}.
\end{equation*}
Finally, discretizing the objective
function~\eqref{eq:distr-heat-cost-function} with the trapezoidal
rule, which is the appropriate discretization of the costs associated
with the space-time discretization that we have chosen, gives
\begin{equation}
  \label{eq:distr-heat-cost-function-discr}
  \frac{\Delta t}{2} \sum_{i=0}^{N-1}
  \costcoeffwaste (\Powerwaste(t_i) + \Powerwaste(t_{i+1}))
  + \costcoeffgas (\Powergas(t_i) + \Powergas(t_{i+1}))
  + \costcoeffpress (\Powerpress(t_i) + \Powerpress(t_{i+1})).
\end{equation}

\subsection{A Space Discretization Scheme based on Central Differences}
\label{sec:centr-discr-scheme}

Since in the discretized optimization problem there is no preferred
space direction, in this section we present an alternative spatial
discretization scheme using central differences.
Later in our numerical results, we then compare this scheme with
the implicit scheme of the last section.

Using the notation of Section~\ref{sec:impl-euler-discr}, \ie,
$t_i$, $i \in \set{0, \dotsc, N}$, for the discrete
time points and $x_{\arc, k}$, $k \in \set{0, \dotsc, M_\arc}$,
for the discrete points in space, we obtain the following discretized
system for $i = 0, \dotsc, N-1$ and $k = 1, \dotsc, M_\arc-1$ that
contains~\eqref{eq:distr-heat-euler-const-density-discr} and
\begin{equation}
  \label{eq:distr-heat-energy-central-discr}
  \begin{split}
    \frac{\temp_\arc(x_{\arc, k}, t_{i+1}) - \temp_\arc(x_{\arc, k},
      t_i)}{\timediff}
    \quad & \\
    +\, \vel_\arc(t_{i+1})
    \frac{\temp_\arc(x_{\arc, k+1}, t_{i+1}) - \temp_\arc(x_{\arc, k-1},
      t_{i+1})}{2 \spacediff}
    \quad & \\
    +\, \frac{4 \heattrans_\arc}{\heatcap \dens \diam_\arc}
    \left( \temp_\arc(x_{\arc, k}, t_{i+1}) - \soiltemp \right)
    & = 0.
  \end{split}
\end{equation}
Because the central difference scheme in
\eqref{eq:distr-heat-energy-central-discr} takes two spatial steps at
a time, we are missing one equation in every timestep.
Therefore, an additional discretization step is needed at the
beginning or the end of the pipe, where we arbitrarily choose the end
of the pipe:
\begin{equation}
  \label{eq:distr-heat-energy-central-last-discr}
  \begin{aligned}
    \frac{\temp_\arc(x_{\arc, M_\arc},t_{i+1}) - \temp_\arc(x_{\arc,
    M_\arc},t_i)}{\timediff}
    \quad & \\
    +\, \vel_\arc(t_{i+1}) \frac{\temp_\arc(x_{\arc, M_\arc}, t_{i+1}) -
      \temp_\arc(x_{\arc, M_\arc - 1}, t_{i+1})}{\spacediffarc}
    \quad & \\
    +\, \frac{4 \heattrans_\arc}{\heatcap \dens
      \diam_\arc}(\temp_\arc(x_{\arc, M_\arc},t_{i+1})-\soiltemp)
    & = 0.
  \end{aligned}
\end{equation}
Note that we do not discretize the continuity equation since it simply
states that velocities only depend on time and not on space.
Finally, the algebraic constraints and the objective function are discretized
as in the last section.


\section{Mixing Models}
\label{sec:mixing-models}

As already mentioned in Section~\ref{sec:modeling}, the mixing model
originally is not well-posed since it is based on arc sets that are
not known a priori.
To handle this issue, we present two different reformulations that we
later compare numerically in Section~\ref{sec:numerical-results}.

\subsection{A Complementarity-Constrained Temperature Mixing Model}
\label{sec:an-mpcc-based}

The sets~$\inflowArcs(\node, t)$ and $\outflowArcs(\node, t)$ used in the
temperature mixing constraints~\eqref{eq:distr-heat-temperature-mixing} of
Problem~\eqref{eq:distr-heat-opt-problem} are not known a priori,
which makes it difficult to use them in an optimization model.
We resolve this problem by replacing them with nonsmooth $\max$-constraints
introduced in \cite{Hante_Schmidt:2017} for a similar setting in gas
transport networks.
The newly introduced variable
\begin{equation}
  \label{eq:distr-heat-positive-mass-flow}
  \posmflow_\arc(t) \define \max\set{0, \mflow_\arc(t)},
  \quad
  \arc \in \Aff \cup \Abf,
\end{equation}
models the positive part of the mass flow~$\mflow_\arc(t)$ of arc
$\arc$.
This is equivalent to
\begin{equation*}
  \label{eq:distr-heat-negative-mass-flow}
  \posmflow_\arc(t) - \mflow_\arc(t) = \max\set{0,
  -\mflow_\arc(t)},
  \quad \arc \in \Aff \cup \Abf.
\end{equation*}
The variable~$\negmflow_\arc(t) \define \posmflow_\arc(t) -
\mflow_\arc(t)$
thus models the negative part of the mass flow~$\mflow_\arc(t)$.
For each node $\node \in \Nodes$ and all $t \in \Tint$, then the following implications are satisfied,
\begin{align*}
  \arc \in \inflowArcs(\node, t) \cap \Inedges(\node) \implies &
  \posmflow_\arc(t) = \mflow_\arc(t), \
  \negmflow_\arc(t) = 0,\\
  \arc \in \outflowArcs(\node, t) \cap \Inedges(\node) \implies &
  \posmflow_\arc(t) = 0, \ \negmflow_\arc(t) =
  -\mflow_\arc(t),\\
  \arc \in \inflowArcs(\node, t) \cap \Outedges(\node) \implies &
  \posmflow_\arc(t) = 0, \ \negmflow_\arc(t) =
  -\mflow_\arc(t),\\
  \arc \in \outflowArcs(\node, t) \cap \Outedges(\node) \implies &
  \posmflow_\arc(t) = \mflow_\arc(t), \
  \negmflow_\arc(t) = 0.
\end{align*}
We can thus reformulate the temperature mixing equations
\eqref{eq:distr-heat-temperature-mixing} at node~$\node \in \Vdh$
without explicitly using the sets~$\inflowArcs(\node, t)$ and
$\outflowArcs(\node, t)$ and obtain
\begin{subequations}
   \label{eq:distr-heat-temperature-mixing-complementary}
  \begin{align}
    \label{eq:distr-heat-temperature-mixing-in-complementary}
    \temp_\node(t) & = \frac{\sum_{\arc \in \Inedges(\node)}
    \posmflow_\arc(t) \temp_{\arc:\node}(t)
    + \sum_{\arc \in \Outedges(\node)}
    \negmflow_\arc(t) \temp_{\arc:\node}(t)}
    {\sum_{\arc \in \Inedges(\node)} \posmflow_\arc(t)
    + \sum_{\arc \in \Outedges(\node)} \negmflow_\arc(t)},
    \\
    \label{eq:distr-heat-temperature-mixing-out-complementary-out}
    0 & = \posmflow_\arc(t)(\temp_{\arc:\node}(t) - \temp_\node(t)),
    & & \arc \in \Outedges(\node),
    \\
    \label{eq:distr-heat-temperature-mixing-out-complementary-in}
    0 & = \negmflow_\arc(t)(\temp_{\arc:\node}(t) - \temp_\node(t)),
    & & \arc \in \Inedges(\node),
  \end{align}
\end{subequations}
for all $t \in \Tint$.
In Lemma~1 of \cite{Hante_Schmidt:2017}, it is shown that
Condition~\eqref{eq:distr-heat-positive-mass-flow} is equivalent to
the complementarity-constrained model
\begin{equation}
  \label{eq:mpcc-max-reform}
  \mflow_\arc(t) = \posmflow_\arc(t) -
  \negmflow_\arc(t),
  \quad \posmflow_\arc(t) \geq 0,
  \quad \negmflow_\arc(t) \geq 0,
  \quad \posmflow_\arc(t) \negmflow_\arc(t) = 0
\end{equation}
for $\node \in \Vdh$ and $\arc \in \nodeArcs(\node)$.
This is a classical mathematical program with complementarity constraints (MPCC) formulation, since for all $\node \in
\Vdh$, $\arc \in \nodeArcs(\node)$, and $t \in \Tint$, the positive
mass flow $\posmflow_\arc(t)$ or the negative mass flow
$\negmflow_\arc(t)$ is equal to zero.
Thus, $\posmflow_\arc(t)$ and $\negmflow_\arc(t)$ form
a complementarity pair.

Using this constraint, we obtain the finite-dimensional MPCC model
\begin{equation}
  \label{eq:distr-heat-opt-problem-discr}
  \begin{split}
    \min \quad & \eqref{eq:distr-heat-cost-function-discr}\\
    \st \quad & \text{Discretized Euler equation: }
    \eqref{eq:distr-heat-euler-const-density-discr}, \\
    & \text{Discretized thermal energy equation: }
    \eqref{eq:distr-heat-energy-discr} \text{ or }
    \eqref{eq:distr-heat-energy-central-discr} \text{ and }
    \eqref{eq:distr-heat-energy-central-last-discr},\\
    & \text{Discretized mass balance equation: }
    \eqref{eq:distr-heat-mass-balance},\\
    & \text{Discretized pressure continuity equations: }
    \eqref{eq:distr-heat-pressure-continuity},\\
    & \text{Discretized temperature mixing equations: }
    \eqref{eq:distr-heat-temperature-mixing-complementary},\\
    & \text{Discretized MPCC max-reformulation: }
    \eqref{eq:mpcc-max-reform},\\
    & \text{Discretized consumer constraints: }
    \eqref{eq:distr-heat-consumer},\\
    & \text{Discretized depot constraints: }
    \eqref{eq:distr-heat-depot},\\
    & \text{Discretized  bounds: }
    \eqref{eq:distr-heat-bounds}
  \end{split}
\end{equation}
for optimizing the control of the district heating network,
which is equivalent to a discretized version of the original
problem~\eqref{eq:distr-heat-opt-problem}.

In general, MPCCs are hard to solve, since they usually do not satisfy
standard constraint qualifications of nonlinear optimization
\cite{Hoheisel_et_al:2013}.
To see this, consider the complementarity
constraints~\eqref{eq:mpcc-max-reform}.
If $\posmflow_\arc(t) = \negmflow_\arc(t) = 0$
holds, \ie, if there is no flow, then the tangential
cone of \eqref{eq:distr-heat-opt-problem-discr} restricted to the
constraints~\eqref{eq:mpcc-max-reform} is nonconvex.
In this case, the tangential cone cannot coincide with the linearized
tangential cone, because the latter cone is always convex.
Thus, the Abadie constraint qualification (ACQ) is not
satisfied; see, \eg, \cite{bonnans2006numerical} for some details on
constraint qualifications.


\subsection{A Nonlinear Programming Based Temperature Mixing Model}
\label{sec:nlp-mixing-model}

Some of our preliminary numerical experiments showed that the MPCC-based
formulation of the mixing model tends to be hard to solve for standard
NLP solvers. For this reason, in this section
we develop a reformulation for which we later demonstrate that it has
better numerical properties.

The thermal energy balance equation in the nodes given by
\begin{equation*}
  \sum_{\arc \in \Inedges(\node)}
  \mflow_\arc(t) \temp_{\arc:\node}(t) \heatcap =
  \sum_{\arc \in \Outedges(\node)}
  \mflow_\arc(t) \temp_{\arc:\node}(t) \heatcap,
  \quad \node \in \Vdh,\, t \in \Tint
\end{equation*}
ensures that no thermal energy is added or lost in the mixing process.
Assuming that the specific heat capacity $\heatcap$ of water is
constant, we can rewrite these equations as
\begin{equation}
  \label{eq:distr-heat-energy-balance-const-cp}
  \sum_{\arc \in \Inedges(\node)}
  \mflow_\arc(t) \temp_{\arc:\node}(t) =
  \sum_{\arc \in \Outedges(\node)}
  \mflow_\arc(t) \temp_{\arc:\node}(t),
  \quad \node \in \Vdh,\, t \in \Tint.
\end{equation}
However, only formulating the thermal energy balance is not sufficient
to get a complete mixing model, since multiple outflow arcs still could
have different temperatures after mixing.
To prevent this, we explicitly include the temperature propagation
equations at the nodes, which equate the temperatures of all outflow
arcs with the mixed node temperature,
\begin{subequations}
  \label{eq:distr-heat-out-temp-propagation}
  \begin{align}
    \label{eq:distr-heat-out-temp-propagation-out}
    \vel_{\arc}(t) \abs{\temp_{\arc:\node}(t) - \temp_\node(t)} & \leq 0,
    & \node \in \Vdh, \ \arc \in \Outedges(\node),\ t \in \Tint,\\
    \label{eq:distr-heat-out-temp-propagation-in}
    \vel_{\arc}(t) \abs{\temp_{\arc:\node}(t) - \temp_\node(t)} & \geq 0,
    & \node \in \Vdh, \ \arc \in \Inedges(\node),\ t \in \Tint.
  \end{align}
\end{subequations}
For $\arc \in \inflowArcs(\node, t)$, these inequalities are always
fulfilled independent of the absolute value of the temperature
difference $\abs{\temp_{\arc:\node}(t) - \temp_\node(t)}$.
For $\arc \in \outflowArcs(\node, t)$, the inequalities are
only satisfied if $\abs{\temp_{\arc:\node}(t) - \temp_\node(t)} = 0$
holds.
See also \cite{Borsche_et_al:2018}, where a similar model is used in a
simulation model with known flow directions.
The following theorem shows that this reformulation is equivalent to the
original one.
\begin{theorem}
  Suppose that all nodes have a positive inflow, \ie,
  \begin{equation*}
    \sum_{\arc \in \inflowArcs(\node, t)} \abs{\mflow_\arc(t)}
    > 0,
    \quad
    \node \in \Nodes.
  \end{equation*}
  Then, the mixing model \eqref{eq:distr-heat-energy-balance-const-cp} and
  \eqref{eq:distr-heat-out-temp-propagation} is an equivalent reformulation of
  the mixing equations \eqref{eq:distr-heat-temperature-mixing}.
\end{theorem}
\begin{proof}
  Let $\node \in \Vdh$.
  We rewrite the mass balance equation
  \eqref{eq:distr-heat-mass-balance} using
  inflow- and outflow-arcs and obtain
  \begin{equation}
    \label{eq:distr-heat-mass-balance-in-out-flow}
    \begin{aligned}
      0 = & \sum_{\arc \in \Inedges(\node)} \mflow_\arc(t)
      - \sum_{\arc \in \Outedges(\node)} \mflow_\arc(t)\\
      = & \left(\sum_{\arc \in \Inedges(\node) \cap \inflowArcs(\node, t)}
        \mflow_\arc(t)
        - \sum_{\arc \in \Outedges(\node) \cap \inflowArcs(\node, t)}
        \mflow_\arc(t)\right)\\
      & + \left(\sum_{\arc \in \Inedges(\node) \cap \outflowArcs(\node, t)}
        \mflow_\arc(t)
        - \sum_{\arc \in \Outedges(\node) \cap \outflowArcs(\node, t)}
        \mflow_\arc(t)\right)\\
      = & \sum_{\arc \in \inflowArcs(\node, t)} \abs{\mflow_\arc(t)}
      - \sum_{\arc \in \outflowArcs(\node, t)} \abs{\mflow_\arc(t)}.
    \end{aligned}
  \end{equation}
  The same ideas applied to the thermal energy balance
  equation~\eqref{eq:distr-heat-energy-balance-const-cp} lead to
  \begin{equation}
    \label{eq:distr-heat-energy-balance-const-cp-in-out-flow}
    \begin{aligned}
      0 = & \sum_{\arc \in \Inedges(\node)}
      \mflow_\arc(t) \temp_{\arc:\node}(t)
      -\sum_{\arc \in \Outedges(\node)}
      \mflow_\arc(t) \temp_{\arc:\node}(t)\\
      = & \left(\sum_{\arc \in \Inedges(\node) \cap \inflowArcs(\node, t)}
        \mflow_\arc(t) \temp_{\arc:\node}(t)
        - \sum_{\arc \in \Outedges(\node) \cap \inflowArcs(\node, t)}
        \mflow_\arc(t) \temp_{\arc:\node}(t)\right)\\
      & + \left(\sum_{\arc \in \Inedges(\node) \cap \outflowArcs(\node, t)}
        \mflow_\arc(t) \temp_{\arc:\node}(t)
        - \sum_{\arc \in \Outedges(\node) \cap \outflowArcs(\node, t)}
        \mflow_\arc(t) \temp_{\arc:\node}(t)\right)\\
      = & \sum_{\arc \in \inflowArcs(\node, t)} \abs{\mflow_\arc(t)}
      \temp_{\arc:\node}(t)
      - \sum_{\arc \in \outflowArcs(\node, t)} \abs{\mflow_\arc(t)}
      \temp_{\arc:\node}(t).
    \end{aligned}
  \end{equation}
  We now assume that the mixing
  equations~\eqref{eq:distr-heat-temperature-mixing} hold.
  Using \eqref{eq:distr-heat-mass-balance-in-out-flow}, we obtain
  \begin{align*}
    0 = & \left(\sum_{\arc \in \inflowArcs(\node, t)} \abs{\mflow_\arc(t)}
          - \sum_{\arc \in \outflowArcs(\node, t)} \abs{\mflow_\arc(t)}\right)
          \temp_\node(t)\\
    = & \left(\sum_{\arc \in \inflowArcs(\node, t)}
        \abs{\mflow_\arc(t)}\right) \temp_\node(t)
        - \left(\sum_{\arc \in \outflowArcs(\node, t)}
        \abs{\mflow_\arc(t)} \temp_\node(t)\right)\\
    = & \left(\sum_{\arc \in \inflowArcs(\node, t)}
        \abs{\mflow_\arc(t)}\right)
        \frac{\sum_{\arc \in \inflowArcs(\node, t)}
        \abs{\mflow_\arc(t)} \temp_{\arc:\node}(t)}
        {\sum_{\arc \in \inflowArcs(\node, t)} \abs{\mflow_\arc(t)}}
        - \sum_{\arc \in \outflowArcs(\node, t)} \abs{\mflow_\arc(t)}
        \temp_{\arc:\node}(t)\\
    = & \sum_{\arc \in \inflowArcs(\node, t)} \abs{\mflow_\arc(t)}
        \temp_{\arc:\node}(t)
        - \sum_{\arc \in \outflowArcs(\node, t)} \abs{\mflow_\arc(t)}
        \temp_{\arc:\node}(t),
  \end{align*}
  which implies the thermal energy balance
  equation~\eqref{eq:distr-heat-energy-balance-const-cp} by using
  \eqref{eq:distr-heat-energy-balance-const-cp-in-out-flow}.

  Consider now an arc~$\arc \in \Outedges(\node)$.
  Then, the temperature propagation
  equation~\eqref{eq:distr-heat-out-temp-propagation-out} is satisfied
  by using~\eqref{eq:distr-heat-temperature-mixing-out},
  \begin{align*}
    \vel_{\arc}(t) \abs{\temp_{\arc:\node}(t) - \temp_\node(t)}
    & = 0
    \quad \text{if } \arc \in \outflowArcs(\node,t),\\
    \underbrace{\vel_{\arc}(t)}_{\leq 0}
    \underbrace{\abs{\temp_{\arc:\node}(t) - \temp_\node(t)}}_{\geq 0}
    & \leq 0 \quad \text{if } \arc \in \inflowArcs(\node,t).
  \end{align*}
  For an arc~$\arc \in \Inedges(\node)$, the temperature propagation
  equation~\eqref{eq:distr-heat-out-temp-propagation-in} is also
  fulfilled
  \begin{align*}
    \vel_{\arc}(t) \abs{\temp_{\arc:\node}(t) - \temp_\node(t)}
    & = 0 \quad \text{if } \arc \in \outflowArcs(\node, t),\\
    \underbrace{\vel_{\arc}(t)}_{\geq 0}
    \underbrace{\abs{\temp_{\arc:\node}(t) - \temp_\node(t)}}_{\geq 0}
    & \geq 0 \quad \text{if } \arc \in \inflowArcs(\node, t),
  \end{align*}
and hence, we have shown the first implication.

  For the reverse implication, we assume that
  \eqref{eq:distr-heat-energy-balance-const-cp} and
  \eqref{eq:distr-heat-out-temp-propagation} hold.
For $\arc \in \outflowArcs(\node, t)$, because of
  \eqref{eq:distr-heat-out-temp-propagation}, we have
  \begin{equation*}
    \begin{aligned}
      \underbrace{\vel_{\arc}(t)}_{>0} \abs{\temp_{\arc:\node}(t) -
        \temp_\node(t)}
      & \leq 0 \quad \text{if } \arc \in \Outedges(\node),\\
      \underbrace{\vel_{\arc}(t)}_{<0} \abs{\temp_{\arc:\node}(t) -
        \temp_\node(t)}
      & \geq 0 \quad \text{if } \arc \in \Inedges(\node).
    \end{aligned}
  \end{equation*}
  Thus, $\abs{\temp_{\arc:\node}(t) - \temp_\node(t)} = 0$ holds, which
  implies \eqref{eq:distr-heat-temperature-mixing-out}.
Then, we use the thermal energy balance
  equation~\eqref{eq:distr-heat-energy-balance-const-cp} to prove
  \eqref{eq:distr-heat-temperature-mixing-in}
  \begin{equation*}
    \begin{aligned}
      0 = & \sum_{\arc \in \Inedges(\node)}
      \mflow_\arc(t) \temp_{\arc:\node}(t)
      -\sum_{\arc \in \Outedges(\node)}
      \mflow_\arc(t) \temp_{\arc:\node}(t)\\
      = & \sum_{\arc \in \inflowArcs(\node, t)} \abs{\mflow_\arc(t)}
      \temp_{\arc:\node}(t)
      - \sum_{\arc \in \outflowArcs(\node, t)} \abs{\mflow_\arc(t)}
      \temp_{\arc:\node}(t)\\
      = & \sum_{\arc \in \inflowArcs(\node, t)} \abs{\mflow_\arc(t)}
      \temp_{\arc:\node}(t)
      - \sum_{\arc \in \outflowArcs(\node, t)}
      \abs{\mflow_\arc(t)} \temp_\node(t)\\
      = & \left(\sum_{\arc \in \inflowArcs(\node, t)}
        \abs{\mflow_\arc(t)}\right)
      \frac{\sum_{\arc \in \inflowArcs(\node, t)}
        \abs{\mflow_\arc(t)} \temp_{\arc:\node}(t)}
      {\sum_{\arc \in \inflowArcs(\node, t)} \abs{\mflow_\arc(t)}}
      - \left(\sum_{\arc \in \outflowArcs(\node, t)}
        \abs{\mflow_\arc(t)}\right) \temp_\node(t)\\
      = & \left(\sum_{\arc \in \inflowArcs(\node, t)}
        \abs{\mflow_\arc(t)}\right)
      \frac{\sum_{\arc \in \inflowArcs(\node, t)}
        \abs{\mflow_\arc(t)} \temp_{\arc:\node}(t)}
      {\sum_{\arc \in \inflowArcs(\node, t)} \abs{\mflow_\arc(t)}}
      - \left(\sum_{\arc \in \inflowArcs(\node, t)}
        \abs{\mflow_\arc(t)}\right) \temp_\node(t)\\
      = & \left(\sum_{\arc \in \inflowArcs(\node, t)}
        \abs{\mflow_\arc(t)}\right)
      \left(\frac{\sum_{\arc \in \inflowArcs(\node, t)}
          \abs{\mflow_\arc(t)} \temp_{\arc:\node}(t)}
        {\sum_{\arc \in \inflowArcs(\node, t)} \abs{\mflow_\arc(t)}}
        - \temp_\node(t)\right).
    \end{aligned}
  \end{equation*}
  Since
  \begin{equation*}
    \sum_{\arc \in \inflowArcs(\node, t)} \abs{\mflow_\arc(t)} > 0
  \end{equation*}
  holds by assumption, the mixing
  equation~\eqref{eq:distr-heat-temperature-mixing-in} follows.
\end{proof}

By introducing a new variable~$\tempdiffvar_{\arc, \node}$ for
all~$\node \in \Vdh$ and~$\arc \in \nodeArcs(\node)$ one can
rewrite~\eqref{eq:distr-heat-out-temp-propagation} to also avoid
absolute values in the equations:
\begin{subequations}
  \label{eq:distr-heat-out-temp-continuity-no-abs}
  \begin{align}
    \label{eq:distr-heat-out-temp-continuity-no-abs-out}
    \vel_{\arc}(t) \tempdiffvar_{\arc, \node}(t) & \leq 0,
    & \node \in \Vdh, \arc \in \Outedges(\node),\\
    \label{eq:distr-heat-out-temp-continuity-no-abs-out-pos}
    \tempdiffvar_{\arc, \node}(t) & \geq \temp_{\arc:\node}(t) - \temp_\node(t),
    & \node \in \Vdh, \arc \in \Outedges(\node),\\
    \label{eq:distr-heat-out-temp-continuity-no-abs-out-neg}
    \tempdiffvar_{\arc, \node}(t) & \geq \temp_\node(t) - \temp_{\arc:\node}(t),
    & \node \in \Vdh, \arc \in \Outedges(\node),\\
    \label{eq:distr-heat-out-temp-continuity-no-abs-in}
    \vel_{\arc}(t) \tempdiffvar_{\arc, \node}(t) & \geq 0,
    & \node \in \Vdh, \arc \in \Inedges(\node),\\
    \label{eq:distr-heat-out-temp-continuity-no-abs-in-pos}
    \tempdiffvar_{\arc, \node}(t) & \geq \temp_{\arc:\node}(t) -
                                    \temp_\node(t),  & \node \in \Vdh, \arc \in \Inedges(\node),\\
    \label{eq:distr-heat-out-temp-continuity-no-abs-in-neg}
    \tempdiffvar_{\arc, \node}(t) & \geq \temp_\node(t) -
                                    \temp_{\arc:\node}(t),  & \node \in \Vdh, \arc \in \Inedges(\node).
  \end{align}
\end{subequations}
We have the following result.
\begin{theorem}
  System~\eqref{eq:distr-heat-out-temp-continuity-no-abs} is feasible
  if and only if the temperature propagation
  equations~\eqref{eq:distr-heat-out-temp-propagation} are feasible.
\end{theorem}
\begin{proof}
  It is easy to see that
  \eqref{eq:distr-heat-out-temp-continuity-no-abs-out-pos} and
  \eqref{eq:distr-heat-out-temp-continuity-no-abs-out-neg} are
  smooth and linear reformulations of
  \begin{equation*}
    \tempdiffvar_{\arc, \node}(t) \geq \abs{\temp_{\arc:\node}(t) -
      \temp_\node(t)},
    \quad \node \in \Vdh, \ \arc \in \Outedges(\node),
  \end{equation*}
  and \eqref{eq:distr-heat-out-temp-continuity-no-abs-in-pos} and
  \eqref{eq:distr-heat-out-temp-continuity-no-abs-in-neg} are  smooth
  and linear reformulations of
  \begin{equation*}
    \tempdiffvar_{\arc, \node}(t) \geq \abs{\temp_{\arc:\node}(t) -
      \temp_\node(t)},
    \quad \node \in \Vdh, \ \arc \in \Inedges(\node).
  \end{equation*}
  Suppose now that \eqref{eq:distr-heat-out-temp-propagation} is
  feasible.   Then,
  \begin{equation*}
    \tempdiffvar_{\arc, \node}(t)
    \define \abs{\temp_{\arc:\node}(t) -
    \temp_\node(t)},
    \quad \node \in \Vdh, \arc \in \Outedges(\node) \cup
    \Inedges(\node),
  \end{equation*}
  satisfy \eqref{eq:distr-heat-out-temp-continuity-no-abs}.

  Next, assume that \eqref{eq:distr-heat-out-temp-continuity-no-abs}
  is feasible. For a node $\node \in \Vdh$ and an outgoing arc
  \mbox{$\arc \in \Outedges(\node)$}, we have
  $\vel_{\arc}(t) \tempdiffvar_{\arc, \node}(t) \leq 0$
  by \eqref{eq:distr-heat-out-temp-continuity-no-abs-out}.
  Thus, either $\vel_{\arc}(t) \leq 0$ or
  $\tempdiffvar_{\arc, \node}(t) = 0$.
  In the first case, it follows that
  \begin{equation*}
    \vel_{\arc}(t) \abs{\temp_{\arc:\node}(t) - \temp_\node(t)} \leq 0.
  \end{equation*}
  In the second case, we obtain that
  \begin{equation*}
    0 \leq \abs{\temp_{\arc:\node}(t) - \temp_\node(t)}
    \leq \tempdiffvar_{\arc, \node}(t) = 0,
  \end{equation*}
  which implies $\temp_{\arc:\node}(t) = \temp_\node(t)$.
  Hence, \eqref{eq:distr-heat-out-temp-propagation-out} is fulfilled.
  The case of a node~$\node \in \Vdh$ and an ingoing arc $\arc \in
  \Inedges(\node)$ can be handled analogously.
\end{proof}
Using the reformulated constraints, we obtain the finite-dimensional
NLP model
\begin{equation}
  \label{eq:distr-heat-opt-problem-discr-nlp}
  \begin{split}
    \min \quad & \eqref{eq:distr-heat-cost-function-discr}\\
    \st \quad & \text{Discretized Euler equation: }
    \eqref{eq:distr-heat-euler-const-density-discr},\\
    & \text{Discretized thermal energy equation: }
    \eqref{eq:distr-heat-energy-discr} \text{ or }
    \eqref{eq:distr-heat-energy-central-discr} \text{ and }
    \eqref{eq:distr-heat-energy-central-last-discr},\\
    & \text{Discretized mass balance equation: }
    \eqref{eq:distr-heat-mass-balance},\\
    & \text{Discretized pressure continuity equations: }
    \eqref{eq:distr-heat-pressure-continuity},\\
    & \text{Discretized thermal energy balance equation: }
    \eqref{eq:distr-heat-energy-balance-const-cp},\\
    & \text{Discretized temperature continuity equations: }
    \eqref{eq:distr-heat-out-temp-continuity-no-abs},\\
    & \text{Discretized consumer constraints: }
    \eqref{eq:distr-heat-consumer},\\
    & \text{Discretized depot constraints: }
    \eqref{eq:distr-heat-depot},\\
    & \text{Discretized bounds: }
    \eqref{eq:distr-heat-bounds}
  \end{split}
\end{equation}
for optimizing the control of the district heating network.

The temperature propagation equations
\eqref{eq:distr-heat-out-temp-continuity-no-abs} still imply a
complementarity structure similar to the complementarity
constraints~\eqref{eq:mpcc-max-reform} of the MPCC-based mixing
model.
In particular, this means that for $\vel_\arc(t) = \tempdiffvar_{\arc,
  \node}(t) = 0$, the tangential cone of
\eqref{eq:distr-heat-opt-problem-discr-nlp} restricted to the
constraints~\eqref{eq:distr-heat-out-temp-continuity-no-abs} is
nonconvex.
In this case, the ACQ is not satisfied, which was also the
case for the formulation discussed in Section~\ref{sec:an-mpcc-based}.
Nevertheless, the reformulation presented in this section results in a
larger tangential cone; see Figure~\ref{fig:tangent-cones}.
Later, in Section~\ref{sec:numerical-results}, we will see that this
gain in constraint regularity can lead to significantly improved
numerical results for some NLP solvers.
\begin{figure}
  \tikzexternalenable
  \tikzsetnextfilename{tikz-imgs/mpcc-tangent-cone}
  \begin{tikzpicture}[scale=1.5,font=\sffamily]


\draw[-] (-0.25,0) -- (1.5,0) node[right] {$\posmflow_{\arc:\node}(t)$};

\draw[-] (0,-0.25) -- (0,1.5) node[above] {$\negmflow_{\arc:\node}(t)$};


\draw[->,color=luh-dark-blue,ultra thick] (0,0) -- (1.5,0);

\draw[->,color=luh-dark-blue,ultra thick] (0,0) -- (0,1.5);

\end{tikzpicture}
%
  \tikzexternaldisable

  \quad
  \tikzexternalenable
  \tikzsetnextfilename{tikz-imgs/nlp-tangent-cone}
  \begin{tikzpicture}[scale=1.5,font=\sffamily]


\shade[shading=axis, bottom color=luh-dark-blue, top color=white, shading 
angle=-45, opacity=0.5]
(0,0) -- (1.5,0) -- (1.5,1.5) -- (0,1.5);


\draw[-] (-1.5,0) -- (1.5,0) node[right] {$\vel_\arc(t)$};

\draw[-] (0,-0.25) -- (0,1.5) node[above] {$\tempdiffvar_{\arc, \node}(t)$};


\draw[->,color=luh-dark-blue,ultra thick] (-1.5,0) -- (1.5,0);

\draw[->,color=luh-dark-blue,ultra thick] (0,0) -- (0,1.5);

\end{tikzpicture}
%
  \tikzexternaldisable

  \caption{\rev{Illustration of the t}angential cones \rev{(thick blue
      axes and} \rev{shaded area}) of the MPCC- (left) and NLP-based (right)
    mixing model.}
  \label{fig:tangent-cones}
\end{figure}
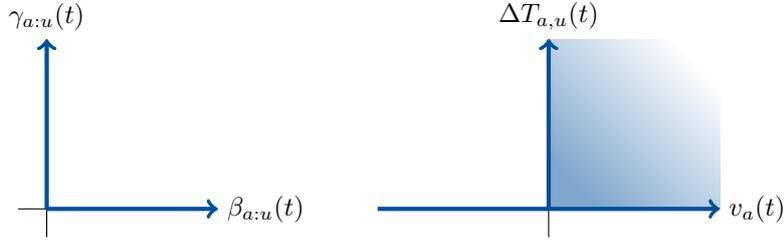



\section{Optimization Techniques}
\label{sec:optimization-techniques}

In this section, we present several optimization techniques that
allow to solve the challenging problem presented and discussed in the
last sections.

\subsection{An Instantaneous Control Approach}
\label{sec:instantaneous-control}

The discretizations described in
Section~\ref{sec:discretized-distr-heat-network} lead to
finite-dimensional but typically very large NLPs or MPCCs.
Since the solution of these problems is very hard in practice, in this
section we develop an instantaneous control approach.
Instantaneous control has been frequently used for challenging control
problems; \cf, \eg, \cite{Choi_et_al:1993,Choi_et_al:1999} for
flow control, and in
\cite{Altmueller_et_al:2010,Hinze2002,Hundhammer2001} for the control
of linear wave equations, of wave equations in networks, or of
vibrating string networks, respectively.
An application to traffic flows can be found
in \cite{Herty_et_al:2007} as well as to mixed-integer nonlinear gas
transport networks models in \cite{Gugat_et_al:2018a}, and for
MPEC-type optimal control problems in~\cite{Antil_et_al:2017}.

The basic idea of instantaneous control is the following.
Starting from the first time period of the discretization and with a
given initial state, we only solve the control problem for this first
time period of our discretized time horizon.
We then apply the resulting control, move one time period
forward in time, solve the control problem restricted to the second
period, etc.
In other words, we solve a series of quasi-stationary problems while
moving forward in time.

This heuristic control approach can be used in two different ways.
First, if successful\rev{, \ie, if an overall feasible control is
  obtained, this} resulting control can be applied directly in
practice.
However, this control typically will be far away from being optimal
for the complete time horizon.
Second, the resulting control can be used to initialize the full NLP
(or MPCC) to obtain a feasible initial point, which usually helps
significantly in solving the overall problem to (local) optimality.

Let us now formally describe the instantaneous control approach.
To this end, we denote the fully discretized problem as
\begin{subequations}
  \label{eq:fully-discretized problem}
  \begin{align}
    \min_x \quad & \sum_{i=1}^N f_i(x_i, x_{i-1})
    \\
    \st \quad & c_i^{\mathcal{E}}(x_i, x_{i-1}) = 0, \quad i = 1, \dotsc, N,
    \\
               & c_i^{\mathcal{I}}(x_i, x_{i-1}) \geq 0, \quad i = 1, \dotsc, N,
    \\
               & d_i^{\mathcal{E}}(x_i) = 0, \quad i = 1, \dotsc, N,
    \\
               & d_i^{\mathcal{I}}(x_i) \geq 0, \quad i = 1, \dotsc, N,
  \end{align}
\end{subequations}
where  $x = (x_i)_{i=0}^N$ and $x_i$ contains all variables
associated to the time point $t_i$.
The super-indices $\mathcal E$, $\mathcal I$ stand for equality and
inequality constraints.
The constraints $c_i^\mathcal{E}$, $c_i^\mathcal{I}$
represent the constraints coupling the time points $t_{i-1}$ and $t_i$ and
$d_i^\mathcal{E}$ and $d_i^\mathcal{I}$ couple all constraints that
only depend on the single time point $t_i$.

Restricted to the time period $[t_{i-1}, t_i]$ and for given $x_{i-1}
= \hat{x}_{i-1}$, this problem can be formulated as
\begin{subequations}
  \label{eq:fully-discretized-problem-one-time-step}
  \begin{align}
    \min_{x_i} \quad & f_i(x_i, \hat{x}_{i-1})
    \\
    \st \quad & c_i^{\mathcal{E}}(x_i, \hat{x}_{i-1}) = 0,
                \quad
                 c_i^{\mathcal{I}}(x_i, \hat{x}_{i-1}) \geq 0,
    \\
               & d_i^{\mathcal{E}}(x_i) = 0,
                 \quad
                 d_i^{\mathcal{I}}(x_i) \geq 0,
  \end{align}
\end{subequations}
With this problem at hand, the instantaneous control method can be
described as in Algorithm~\ref{alg:inst-control}.
\begin{algorithm}
  \caption{Instantaneous Control Algorithm}
  \label{alg:inst-control}
  \begin{algorithmic}[1]
    \REQUIRE The original problem, a discretized time
    horizon~$\set{t_0, \dotsc, t_N}$,
    a full discretization of the problem,
    and initial conditions $x_0 = \hat{x}_0$.
    \FOR{$i = 1, \dotsc, N$}
    \STATE \label{alg:inst-control:solve-problem} Solve the
    problem~\eqref{eq:fully-discretized-problem-one-time-step} for time
    step~$i$ and variables~$x_{i-1}$ fixed to $\hat{x}_{i-1}$.
    \STATE Denote the optimal solution by $\hat{x}_{i}$.
    \ENDFOR
  \end{algorithmic}
\end{algorithm}

Note that this approach is usually very fast in practice because the
variables~$x_i$ in the
NLP~\eqref{eq:fully-discretized-problem-one-time-step} can be
reasonably initialized with the values~$\hat{x}_{i-1}$.
\rev{Note again that if Algorithm~\ref{alg:inst-control} is successful,
\ie, if every problem in Line~\ref{alg:inst-control:solve-problem} is
solved, the method results in an overall feasible control for the
entire time horizon.}

\subsection{Penalty Formulations}
\label{sec:penalty-form}

In this section, we consider the fully discretized
version~\eqref{eq:fully-discretized problem} of our problem.
This problem is mainly governed by equality constraints from physics
and has rather few controls.
Thus, it contains only very few degrees of freedom, which renders
the problem hard to solve in practice\rev{; see, \eg,
\cite{Schmidt_et_al:2016}, where the same phenomenon is discussed for
the case of nonlinear gas network optimization models}.
One possible remedy in such situations is to consider
the relaxed version
\begin{subequations}
  \label{eq:fully-discretized problem-w-slacks}
  \begin{align}
    \min_{x,s \geq 0} \quad & \sum_{i=1}^N f_i(x_i, x_{i-1}) + \norm{Ws}
    \\
    \st \quad & c_i^{\mathcal{E}}(x_i, x_{i-1}) + s_i^{\mathcal{E},c,+}
                - s_i^{\mathcal{E},c,-} = 0,
                & i = 1, \dotsc, N,
    \\
                 & c_i^{\mathcal{I}}(x_i, x_{i-1}) +
                   s_i^{\mathcal{I},c,+} \geq 0,
                   & i = 1, \dotsc, N,
    \\
                 & d_i^{\mathcal{E}}(x_i) + s_i^{\mathcal{E},d,+} -
                   s_i^{\mathcal{E},d,-} = 0,
              & i = 1, \dotsc, N,
    \\
                 & d_i^{\mathcal{I}}(x_i) + s_i^{\mathcal{I},d,+} \geq
                   0,
                   & i = 1, \dotsc, N.
  \end{align}
\end{subequations}
Here, every equality constraint~$c_i^{\mathcal{E}}$ is equipped with a
slack variable~$s_i^{\mathcal{E},c,+}$ for the negative and a slack
variable~$s_i^{\mathcal{E},c,-}$ for the positive violation of the
constraint.
Obviously, inequality constraints only require slack variables for
their negative violation and the constraints~$d$ are handled in the
same way.
The vector~$s$ in the objective function then denotes the vector of
all slack variables used in the constraints and the matrix~$W$ is a
diagonal matrix with positive diagonal entries representing scaling
factors for the respective slack variables.
Obviously, a solution with $s=0$ is also a solution of the original
problem.

We also combine the penalty formulation with the instantaneous control
approach described in the last section.
In practice, it may happen that a sub-problem in the
for-loop of Algorithm~\ref{alg:inst-control} cannot be solved to a
feasible point.
Thus, we also introduce a corresponding penalty formulation in every
iteration of the instantaneous control algorithm.
If, in an iteration, the slack variables are too large,
then we consider the constraint violations of the infeasible
point (for the original problem) and increase the respective weights
in~$W$ in order to penalize the violation of the most violated
constraints even stronger.
Then, the sub-problem is solved again and the process is repeated until the
sub-problem is solved to feasibility (or a maximum number of
re-iterations is reached).
Finally note that it is often preferable in practice to not equip all
constraints with slack variables but only a subset of constraints,
\eg, all nonlinear constraints.
See \cite{Koch_et_al:ch11:2015} for a detailed discussion of relaxed
penalty models in the related field of gas network optimization.

\subsection{A Preprocessing Technique for Fixing Flow Directions}
\label{sec:flow-dir-presolve}

Due to their complementarity structure, the temperature mixing
equations of the MPCC-based mixing model as well as of the NLP-based
mixing model usually lead to difficulties in the solution process. To
avoid these difficulties, we first identify nodes with incident arcs
on which the flow direction is known, which helps to reduce the
hardness of the model.
In addition to simplifying the mixing equations, one can also smoothen
the friction term
\begin{equation*}
  \lambda_\arc \frac{\abs{\vel_\arc}
    \vel_\arc \dens_\arc}{2\diam_\arc}(x,t)
\end{equation*}
in the momentum equation~\eqref{eq:distr-heat-euler-momentum-const-density} if
the sign of the velocity~$\vel_\arc$ is known a priori.
This leads to a simple but powerful preprocessing strategy to identify
arcs with fixed flow direction in Algorithm~\ref{alg:flow-dir-presolve}.
\begin{algorithm}
  \caption{Flow Direction Presolve}
  \label{alg:flow-dir-presolve}
  \begin{algorithmic}[1]
    \REQUIRE The graph $\Graph = (\Vertices, \Arcs)$ of the district heating
    network.
    \ENSURE Sets $\Apos$ and $\Aneg$ only containing arcs with fixed
    positive flow direction or fixed negative flow direction, respectively.
    \STATE Set $\Apos \define \Ac \cup \set{\ad}$ and $\Aneg \define
    \emptyset$.
    \STATE Consider the undirected graph $\hat{\Graph} = (\Vertices, \Arcs
    \setminus \Apos)$.
    \STATE Find all 2-edge-connected components of $\hat{\Graph}$.
    \STATE Contract every 2-edge-connected component in $\hat{\Graph}$ to a
    single node, yielding
    a forest, because the bridge arcs are the only
    arcs that remain in $\hat{\Graph}$, so that all flow directions in
    $\hat{\Graph}$ are known.
    \STATE Assign all arcs in $\hat{\Graph}$ to the sets $\Apos$ or $\Aneg$
    using depth-first search starting in~$\node$ and~$\otherNode$ for $\ad =
    (\node, \otherNode)$.
    \RETURN $\Apos$ and $\Aneg$.
  \end{algorithmic}
\end{algorithm}
The idea behind Algorithm~\ref{alg:flow-dir-presolve} is to return
the depot arc, all consumer arcs, and all arcs that are not contained in a
cycle.
Some arcs in cycles can also have a fixed flow direction as well.
To detect such arcs, other algorithms would be needed, which we do not
discuss.

Given the result of Algorithm~\ref{alg:flow-dir-presolve},
the velocity $\vel_\arc$ and mass flow~$\mflow_\arc$ of arcs~$\arc$ in~$\Apos$
or~$\Aneg$ can be bounded by zero from below or above, respectively.
\begin{align*}
  \vel_\arc \geq 0, \quad \mflow_\arc & \geq 0, \quad \arc \in \Apos,\\
  \vel_\arc \leq 0, \quad \mflow_\arc & \leq 0, \quad \arc \in \Aneg.
\end{align*}
Additionally, all friction terms in the momentum equations can be
reformulated as
\begin{align*}
  \lambda_\arc \frac{\abs{\vel_\arc}
    \vel_\arc \dens_\arc}{2\diam_\arc}
  & = 0, \quad \arc \in \left(\Aff \cup \Abf\right) \setminus \left(\Apos \cup
    \Aneg\right),
  \\
  \lambda_\arc \frac{\vel_\arc^2 \dens_\arc}{2\diam_\arc}
  & = 0, \quad \arc \in \left(\Aff \cup \Abf\right) \cap \Apos,\\
  - \lambda_\arc \frac{\vel_\arc^2 \dens_\arc}{2\diam_\arc}
  & = 0, \quad \arc \in \left(\Aff \cup \Abf\right) \cap \Apos,
\end{align*}
where,  for better readability,  we have omitted the dependence on~$x$
and $t$.
In this way, the friction terms are smoothed for all arcs~$\arc \in \Apos
\cup \Aneg$.

Consider now the MPCC-based mixing model.
For arcs $\arc \in \Apos$, one can fix the variable for the negative
part of the mass as~$\negmflow_\arc$ to $0$ and for arcs $\arc \in
\Aneg$, one can fix the variable for positive part of the mass flow
$\posmflow_\arc$ to $0$.
The MPCC-based mixing equation
\eqref{eq:distr-heat-temperature-mixing-in-complementary} then turns into
\begin{equation*}
  \temp_\node(t)  = \frac{\sum_{\arc \in \Inedges(\node) \setminus \Aneg}
  \posmflow_\arc(t) \temp_{\arc:\node}(t)
  + \sum_{\arc \in \Outedges(\node) \setminus \Apos}
  \negmflow_\arc(t) \temp_{\arc:\node}(t)}
{\sum_{\arc \in \Inedges(\node) \setminus \Aneg} \posmflow_\arc(t)
  + \sum_{\arc \in \Outedges(\node) \setminus \Apos} \negmflow_\arc(t)},
\quad t \in \Tint,
\end{equation*}
and \eqref{eq:distr-heat-temperature-mixing-out-complementary-out} and
\eqref{eq:distr-heat-temperature-mixing-out-complementary-in} can be
simplified to
\begin{align*}
  0 & = \posmflow_\arc(t)(\temp_{\arc:\node}(t) - \temp_\node(t)),
  & & \arc \in \Outedges(\node) \setminus \left(\Apos \cup \Aneg\right),\, t
  \in
  \Tint,
  \\
  0 & = \temp_{\arc:\node}(t) - \temp_\node(t),
  & & \arc \in \Outedges(\node) \cap \Apos,\, t \in \Tint,
  \\
  0 & = \negmflow_\arc(t)(\temp_{\arc:\node}(t) - \temp_\node(t)),
  & & \arc \in \Inedges(\node) \setminus \left(\Apos \cup \Aneg\right),\, t \in
  \Tint,
  \\
  0 & = \temp_{\arc:\node}(t) - \temp_\node(t),
  & & \arc \in \Inedges(\node) \cap \Aneg,\, t \in \Tint.
\end{align*}
This means that for $\arc \in \Apos$,
Equation~\eqref{eq:distr-heat-temperature-mixing-out-complementary-in}
is not needed any more and for $\arc \in \Aneg$,
Equation~\eqref{eq:distr-heat-temperature-mixing-out-complementary-out}
can be removed.
Thus, every MPCC-mixing equation that contains an arc in $\Apos$ or
$\Aneg$ either gets simplified or is dropped.
Moreover, the number of nonlinearities is reduced as well.

Similarly, for the NLP-based mixing model, we can simplify the
temperature propagation
equations~\eqref{eq:distr-heat-out-temp-continuity-no-abs} as
\begin{align*}
  \vel_{\arc}(t) \tempdiffvar_{\arc, \node}(t) & \leq 0,
  & & \node \in \Vdh, \arc \in \Outedges(\node) \setminus \left(\Apos \cup
  \Aneg\right),\\
  \tempdiffvar_{\arc, \node}(t) & \geq \temp_{\arc:\node}(t) - \temp_\node(t),
  & & \node \in \Vdh, \arc \in \Outedges(\node) \setminus \left(\Apos \cup
  \Aneg\right),\\
  \tempdiffvar_{\arc, \node}(t) & \geq \temp_\node(t) - \temp_{\arc:\node}(t),
  & & \node \in \Vdh, \arc \in \Outedges(\node) \setminus \left(\Apos \cup
  \Aneg\right),\\
  \temp_{\arc:\node}(t) - \temp_\node(t) & = 0,
  & & \node \in \Vdh, \arc \in \Outedges(\node) \cap \Apos,\\
  \vel_{\arc}(t) \tempdiffvar_{\arc, \node}(t) & \geq 0,
  & & \node \in \Vdh, \arc \in \Inedges(\node) \setminus \left(\Apos \cup
  \Aneg\right),\\
  \tempdiffvar_{\arc, \node}(t) & \geq \temp_{\arc:\node}(t) -
  \temp_\node(t),  & & \node \in \Vdh, \arc \in \Inedges(\node) \setminus
  \left(\Apos \cup \Aneg\right),\\
  \tempdiffvar_{\arc, \node}(t) & \geq \temp_\node(t) -
  \temp_{\arc:\node}(t),  & & \node \in \Vdh, \arc \in \Inedges(\node)
  \setminus \left(\Apos \cup \Aneg\right),\\
  \temp_{\arc:\node}(t) - \temp_\node(t) & = 0,
  & & \node \in \Vdh, \arc \in \Inedges(\node) \cap \Aneg.
\end{align*}
Again, all equations in
\eqref{eq:distr-heat-out-temp-continuity-no-abs} that are defined on
arcs in $\Apos$ or $\Aneg$ either get simplified or dropped.
The thermal energy balance
equation~\eqref{eq:distr-heat-energy-balance-const-cp} remains
unchanged.

\subsection{\rev{Initial Conditions}}
\label{sec:init-term-cond}

To compute a good and realistic control of the district heating network,
physically reasonable initial conditions are required.
To obtain such conditions, we compute a stationary solution of the
network for the first time step.
The stationary model we use is the same as our standard model at~$t=0$,
except that all time derivatives are zero.
In this case, the Euler momentum
equation~\eqref{eq:distr-heat-euler-momentum-const-density} becomes
\begin{align*}
  \frac{\press_\arc(\length_\arc,0) - \press_\arc(0,0)}{\length_\arc}
  = - \grav \dens h'_\arc
  - \lambda_\arc \frac{\abs{\vel_\arc}
    \vel_\arc \dens}{2\diam_\arc}(0),
  \quad \arc \in \Aff \cup \Abf,
\end{align*}
and the thermal energy equation \eqref{eq:distr-heat-energy} becomes
\begin{equation*}
\vel_\arc(0) \dparlong{x}{\temp_\arc}(x,0)
+ \frac{4 \heattrans_\arc}{\heatcap \dens_\arc(x,0)
  \diam_\arc}(\temp_\arc(x,0)-\soiltemp)
= 0,
\quad \arc \in \Aff \cup \Abf.
\end{equation*}
All algebraic equations stay the same but are only considered at~$t=0$.
The solution of this stationary model is then used to identify the initial
conditions.


\section{Numerical Results}
\label{sec:numerical-results}

In this section, we present and discuss numerical results for the
models and techniques introduced in the previous sections.
The models have been formulated using \GAMS~25.1.2~\cite{McCarl:2009}.
The resulting instances are solved using the solvers
\Ipopt~3.12~\cite{Waechter_Biegler:2006},
\KNITRO~10.3.0~\cite{Byrd_et_al:2006},
\CONOPTfour~4.06~\cite{Drud:1994,Drud:1995,Drud:1996}, and
\SNOPT~7.2-12.1~\cite{Gill_et_al:2005}.
We apply our technique to two different realistic district heating
networks; the so-called \AROMA network given in \Figref{fig:aroma} and
the so-called \STREET network given in \Figref{fig:street}.
\begin{figure}
  \tikzexternalenable
  \tikzsetnextfilename{tikz-imgs/aroma}
  \begin{tikzpicture}
  \def \radius {11pt}
  \def \lineThickness {very thick}
  \def \nodes {0.0/0.0/VL0, 3.3333333333333335/0.0/VL1, 
  4.666666666666667/-1.3333333333333333/VL2, 8.0/-1.3333333333333333/VL3, 
  9.333333333333334/0.0/VL4, 12.0/0.0/VL5, 8.0/1.3333333333333333/VL7, 
  12.0/1.3333333333333333/VL8, 4.666666666666667/1.3333333333333333/VL6, 
  3.3333333333333335/0.0/VL1}
  \def \pipes {VL0/VL1, VL1/VL2, VL2/VL3, VL3/VL4, VL4/VL5, VL4/VL7, VL7/VL8, 
  VL6/VL7, VL1/VL6}
  
  \foreach \x/\y/\name in \nodes
  \node[circle, draw, minimum size=\radius](\name) at (\x,\y){ };
  
  \foreach \s/\t in \pipes
  \draw [->, \lineThickness] (\s) -- (\t);
\end{tikzpicture}
  \tikzexternaldisable

  \caption{The forward-flow part of the \AROMA network}
  \label{fig:aroma}
\end{figure}
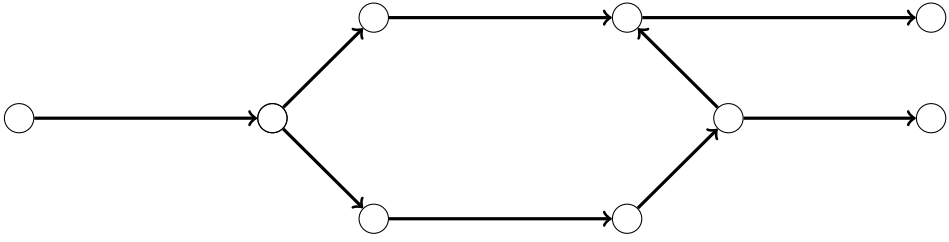
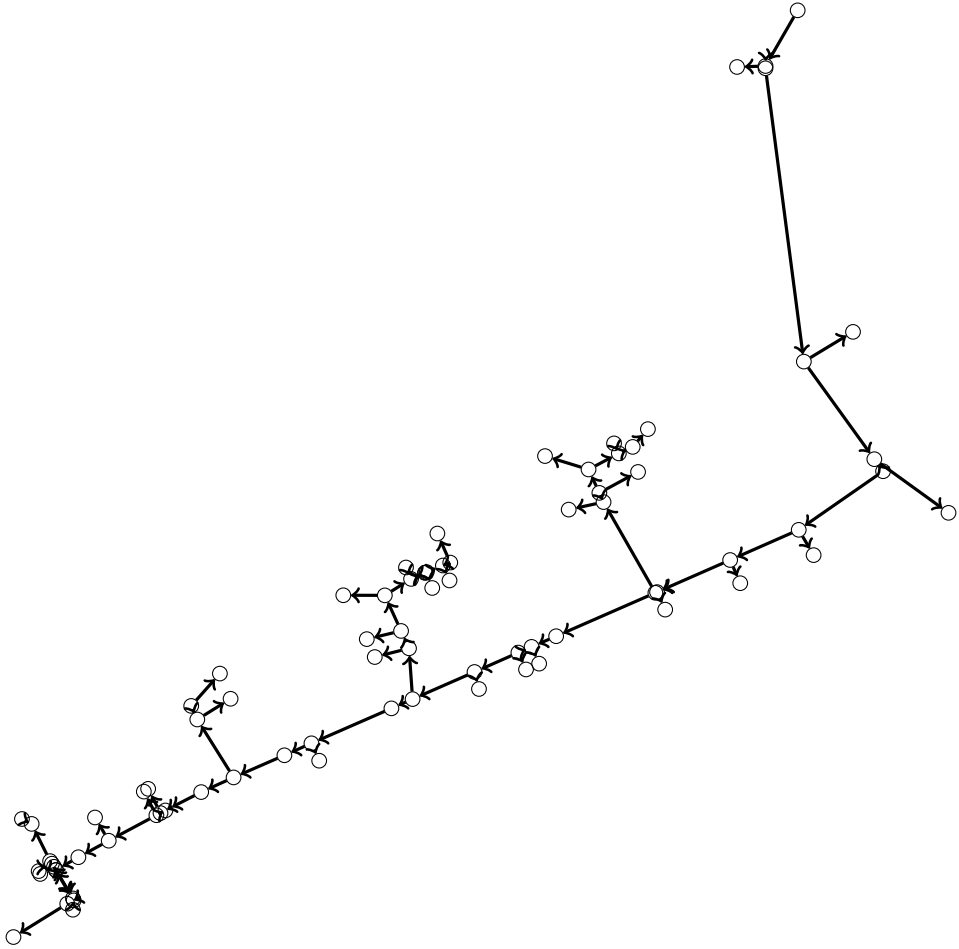
\begin{figure}
  \tikzexternalenable
  \tikzsetnextfilename{tikz-imgs/street}
  \begin{tikzpicture}
    \def \radius {11pt}
    \def \lineThickness {very thick}
    \def \nodes {5.080064999999999/8.221045/n35, 
    4.65738/7.482640000000001/n145, 
    4.65668/7.45285/n78, 5.161365/3.563695/n95, 6.088745/2.269805/n58, 
    6.2039/2.10992/n142, 5.09408/1.33328/n121, 4.191585/0.9329850000000001/n29, 
    3.229815/0.507705/n36, 3.2079700000000004/0.49821/n19, 
    1.902235/-0.07736499999999999/n1, 1.57901/-0.21841/n38, 
    1.40543/-0.29428/n64, 
    1.50589/-0.51829/n143, 0.828155/-0.5514300000000001/n49, 
    0.8868199999999999/-0.7772/n138, 
    0.013425000000000001/-0.9110849999999999/n28, 
    -0.26522/-1.034035/n152, -1.317405/-1.497105/n136, 
    -1.6733149999999999/-1.6560400000000002/n123, 
    -2.341955/-1.9497250000000002/n105, -2.819515/-1.18163/n33, 
    -2.38152/-0.9067299999999999/n8, -2.89958/-1.0012050000000001/n2, 
    -2.52215/-0.57273/n32, -2.7681899999999997/-2.14395/n67, 
    -3.23904/-2.3874/n7, 
    -3.3037900000000002/-2.42255/n161, -3.46609/-2.100765/n79, 
    -3.3577399999999997/-2.4508/n87, -3.523035/-2.137715/n26, 
    -3.9847650000000003/-2.7886949999999997/n69, 
    -4.165335000000001/-2.48086/n4, 
    -4.38421/-3.00761/n30, -4.679295/-3.177245/n45, -4.70498/-3.13436/n92, 
    -4.73237/-3.093075/n111, -4.7561599999999995/-3.05904/n85, 
    -4.446035/-3.5794849999999996/n90, -4.530505/-3.624015/n104, 
    -5.238185/-4.066115/n63, -4.455685/-3.553695/n15, -4.45612/-3.703785/n102, 
    -4.679295/-3.177245/n45, -4.999015/-2.565885/n82, -5.126635/-2.5006/n77, 
    -4.90691/-3.191845/n3, -4.886305/-3.229555/n81, -1.2156/-1.728005/n75, 
    -0.033755/-0.238305/n31, -0.48441/-0.353475/n126, -0.13841/-0.007245/n107, 
    -0.5893900000000001/-0.11597500000000001/n157, 
    -0.352025/0.46426999999999996/n150, 
    -0.0067599999999999995/0.6823349999999999/n127, 
    0.17537/0.7626000000000001/n160, 0.183115/0.76605/n61, 
    0.19856000000000001/0.772845/n156, 0.41187500000000005/0.86026/n76, 
    0.5089349999999999/0.8990449999999999/n153, 0.33947/1.28417/n141, 
    0.49917999999999996/0.661805/n60, 0.271625/0.562125/n12, 
    -0.07361/0.833855/n66, 
    -0.898385/0.46573000000000003/n70, 1.67822/-0.44127/n73, 
    2.525945/1.70098/n44, 
    2.4709700000000003/1.82303/n122, 2.328475/2.1333100000000003/n43, 
    1.75576/2.31108/n106, 2.72697/2.346025/n99, 2.66403/2.479405/n52, 
    2.9095500000000003/2.43264/n20, 3.109475/2.668565/n91, 
    2.97991/2.1008199999999997/n124, 2.06732/1.6012049999999998/n57, 
    3.3368/0.27531/n88, 4.324479999999999/0.625475/n65, 
    5.289915/0.9977550000000001/n14, 7.064964999999999/1.55869/n132, 
    5.808285/3.957945/n137, 4.282525/7.469965/n103}
    \def \pipes {n35/n145, n145/n78, n78/n95, n95/n58, n58/n142, n142/n121, 
    n121/n29, n29/n36, n36/n19, n19/n1, n1/n38, n38/n64, n64/n143, n64/n49, 
    n49/n138, n49/n28, n28/n152, n152/n136, n136/n123, n123/n105, n105/n33, 
    n33/n8, 
    n33/n2, n2/n32, n105/n67, n67/n7, n7/n161, n161/n79, n7/n87, n87/n26, 
    n87/n69, 
    n69/n4, n69/n30, n30/n45, n45/n92, n92/n111, n111/n85, n85/n90, n90/n104, 
    n104/n63, n15/n90, n15/n102, n45/n15, n85/n82, n82/n77, n111/n3, n92/n81, 
    n136/n75, n28/n31, n31/n126, n31/n107, n107/n157, n107/n150, n150/n127, 
    n127/n160, n160/n61, n61/n156, n156/n76, n76/n153, n153/n141, n76/n60, 
    n156/n12, n127/n66, n150/n70, n38/n73, n19/n44, n44/n122, n122/n43, 
    n43/n106, 
    n43/n99, n99/n52, n99/n20, n20/n91, n122/n124, n44/n57, n36/n88, n29/n65, 
    n121/n14, n58/n132, n95/n137, n145/n103}
    
    \foreach \x/\y/\name in \nodes
    \node[circle, draw, minimum size=\radius, scale=0.5](\name) at (\x,\y){ };
    
    \foreach \s/\t in \pipes
    \draw [->, \lineThickness] (\s) -- (\t);
  \end{tikzpicture}
  \tikzexternaldisable

  \caption{The forward-flow part of the \STREET network}
  \label{fig:street}
\end{figure}
The \AROMA network consists of $18$~nodes, $24$~arcs ($1$~depot,
$5$~consumers, and $18$~pipes), and one cycle
each in the forward-flow and the backward-flow network.
Its total pipe length is $\SI{7262.4}{\meter}$.
The \STREET network is a part of a real-world district
heating network with $162$~nodes, $195$~arcs ($1$~depot,
$32$~consumers, and $162$~pipes), and a total pipe length of
$\SI{7627.106}{\meter}$.
\rev{Both networks contain a cycle.
  Thus, not all flow directions are known in advance.
  The preprocessing technique described in
  Section~\ref{sec:flow-dir-presolve} can fix the flow directions for
  6 out of the 18~pipes of the \AROMA network and for 150 out of the
  162~pipes for the \STREET network.
  The larger number of fixations for the \STREET network follows
  from the fact that it only contains a small cycle whereas the major
  part of the network is tree-shaped.}

\rev{Let us also note that we used the $\ell_1$ norm throughout this
  section for the penalty terms in~\eqref{eq:fully-discretized
    problem-w-slacks}.}

The remainder of this section is split up into two parts.
In \Secref{sec:comp-model-vari}, we compare different variants of our
model (namely the MPCC- and the NLP-based mixing model as well as the
two different discretization schemes for the PDEs) and different NLP
solvers.
In \Secref{sec:optim-depot-contr} we then discuss properties of
optimized heat and flow controls at the depot for the \AROMA and the
\STREET network.

\subsection{Comparison of Model Variants and NLP Solvers}
\label{sec:comp-model-vari}

We now compare the performance of different NLP solvers
applied to the two different spatial discretization schemes (the
implicit Euler and the scheme based on central differences) as well as
the two mixing models (the MPCC- and the NLP-based model).
To this end, we consider the \AROMA network with a time horizon of one
day equipped with a time discretization using $30$~minute intervals.
The stepsize of the spatial discretization is $\SI{150}{\meter}$.

The numerical results are given in \Tabref{tab:aroma-1800-150}.
\begin{table}
  \centering
  \caption{Numerical results for all combination of model variants and
    NLP solvers for the \AROMA network with $\timediff =
    \SI{1800}{\second}$ and $\spacediffarc = \SI{150}{\meter}$.}
  \label{tab:aroma-1800-150}
  \resizebox{\columnwidth}{!}{\begin{tabular}{rrrrrrrrrrrrrr}
\toprule
Mixing & Disrc. & $t$ (all) & $t$ (NLP) & $t$ (IC) & \#IC & Mean & Median &
Min. & Max. & $t$ (stat) & \#stat & Obj. & Cost\\
\midrule
\multicolumn{14}{c}{\CONOPTfour}\\
\midrule
MPCC & Centr. diff. & $82.593$ & $80.731$ & $1.862$ & $60$ & $0.031$ & $0.030$ & $0.028$ & $0.042$ & $0.295$ & $4$ & $97.863$ & $28.318$\\
MPCC & Impl. Euler & $16.152$ & $13.742$ & $2.410$ & $78$ & $0.031$ & $0.030$ & $0.027$ & $0.046$ & $0.200$ & $3$ & $275.101$ & $57.291$\\
NLP & Centr. diff. & $16.732$ & $14.910$ & $1.822$ & $60$ & $0.030$ & $0.030$ & $0.028$ & $0.043$ & $0.218$ & $2$ & $119.036$ & $26.416$\\
NLP & Impl. Euler & $18.828$ & $16.634$ & $2.194$ & $72$ & $0.030$ & $0.029$ & $0.028$ & $0.041$ & $0.131$ & $1$ & $100.812$ & $27.647$\\
\midrule
\multicolumn{14}{c}{\Ipopt}\\
\midrule
MPCC & Centr. diff. & $272.993$ & $261.890$ & $11.103$ & $56$ & $0.198$ & $0.166$ & $0.101$ & $0.954$ & $2.721$ & $2$ & $65.405$ & $23.203$\\
MPCC & Impl. Euler & $447.746$ & $431.985$ & $15.761$ & $79$ & $0.200$ & $0.148$ & $0.110$ & $1.135$ & $2.842$ & $2$ & $91.031$ & $49.329$\\
NLP & Centr. diff. & $326.375$ & $319.468$ & $6.907$ & $51$ & $0.135$ & $0.126$ & $0.087$ & $0.250$ & $0.262$ & $1$ & $62.012$ & $47.153$\\
NLP & Impl. Euler & $242.349$ & $68.533$ & $173.816$ & $98$ & $1.774$ & $0.229$ & $0.116$ & $62.834$ & $0.313$ & $1$ & $198.393$ & $53.842$\\
\midrule
\multicolumn{14}{c}{\KNITRO}\\
\midrule
MPCC & Centr. diff. & $933.063$ & $900.254$ & $32.809$ & $74$ & $0.443$ & $0.104$ & $0.046$ & $22.293$ & $0.142$ & $1$ & $42.013$ & $21.018$\\
MPCC & Impl. Euler & $925.191$ & $900.289$ & $24.902$ & $83$ & $0.300$ & $0.147$ & $0.072$ & $6.032$ & $1.527$ & $1$ & --- & $18.170$\\
NLP & Centr. diff. & $61.115$ & $57.636$ & $3.479$ & $50$ & $0.070$ & $0.068$ & $0.044$ & $0.109$ & $0.056$ & $1$ & $44.570$ & $43.987$\\
NLP & Impl. Euler & $38.068$ & $32.102$ & $5.966$ & $71$ & $0.084$ & $0.069$ & $0.048$ & $0.430$ & $0.187$ & $1$ & --- & $57.043$\\
\midrule
\multicolumn{14}{c}{\SNOPT}\\
\midrule
MPCC & Centr. diff. & $25.146$ & $23.621$ & $1.525$ & $48$ & $0.032$ & $0.030$ & $0.026$ & $0.071$ & $0.116$ & $2$ & $51.592$ & $46.923$\\
MPCC & Impl. Euler & $24.639$ & $20.852$ & $3.787$ & $114$ & $0.033$ & $0.032$ & $0.025$ & $0.053$ & $0.138$ & $2$ & $195.360$ & $54.798$\\
NLP & Centr. diff. & $44.661$ & $42.024$ & $2.637$ & $71$ & $0.037$ & $0.035$ & $0.028$ & $0.060$ & $0.062$ & $1$ & $182.363$ & $55.140$\\
NLP & Impl. Euler & $45.769$ & $43.914$ & $1.855$ & $53$ & $0.035$ & $0.033$ & $0.028$ & $0.072$ & $0.067$ & $1$ & --- & $50.311$\\
\bottomrule
\end{tabular}}
\end{table}
The columns of the table contain the following information.
\begin{description}
\item[Mixing] The mixing model; MPCC-based
  (\Secref{sec:an-mpcc-based}) or NLP-based
  (\Secref{sec:nlp-mixing-model}).
\item[Discr.] The implicit Euler discretization
  (\Secref{sec:impl-euler-discr}) or the discretization based on
  central differences (\Secref{sec:centr-discr-scheme}).
\item[$t$ (all)] The overall solution time including the initial value
  computation using the instantaneous control approach
  (\Secref{sec:instantaneous-control}), the presolve step to fix flow
  directions (\Secref{sec:flow-dir-presolve}), and the computation of
  the initial physical state (\Secref{sec:init-term-cond}).
  All running times in the table are given in seconds.
\item[$t$ (NLP)] The time to solve the NLP on the entire time horizon,
  which is initialized with the solution of the instantaneous
  control approach.
\item[$t$ (IC)] The time required to apply the instantaneous control
  approach.
\item[\#IC] The total number of instantaneous control steps including
  re-iterations applied if the scaled max-norm of all slack values
  exceeds the tolerance of~$10^{-2}$.
\item[Mean, Median, Min. Max.] The mean, median, minimum, and maximum
  time of all (re-)iterations of the instantaneous control approach.
\item[$t$ (stat)] The time required to compute the stationary solution
  that is used as an initial physical state.
\item[\#stat] The required number of re-iterations for computing the
  stationary solution.
\item[Obj.] The objective function value of the problem, which is the
  sum of the control costs and the scaled penalty terms.
  Here, \enquote{---} means that the final value of the max-norm of
  all scaled slack values exceeds the tolerance of $10^{-2}$.
\item[Cost] The control costs part of the objective function value;
  see \eqref{eq:distr-heat-cost-function-discr}.
\end{description}

If we first consider the overall time required to solve the problem
(\enquote{$t$ (all)}), we see that the results are highly
heterogeneous \wrt the chosen NLP solver.
The fastest approach (\SI{16.732}{\second}) is obtained by \CONOPTfour
applied to the MPCC-based mixing model and the implicit Euler
discretization.
In contrast, \KNITRO applied to the MPCC-based mixing model and the
discretization scheme based on central differences takes
\SI{933.063}{\second}, which corresponds to a factor larger then 55.
Since every solver gets exactly the same models to be solved, this
strongly indicates the hardness of the district heating network
optimization problems.

It also strongly depends on the chosen solver whether the MPCC- or the
NLP-based mixing model is used.
For instance, \KNITRO performs very poor on the MPCC-based model and
significantly benefits from the NLP-based reformulation.
On the other hand, for \SNOPT it is exactly the other way around
(although the difference in solution times is not as drastic as for
\KNITRO).
The choice of the discretization scheme for the PDEs does not
influence the solution times significantly.
However, it may influence how the solvers are able to reduce the
penalty terms in the objective function; see, \eg, \KNITRO, which is
not able to reduce the penalty terms so that the max-norm of all
scaled slack values is below $10^{-2}$ if the implicit Euler scheme is used.
A comparable behavior can also be seen in the instantaneous control
approach:
All solvers require more re-iterations to reduce the
penalty terms for the implicit Euler discretization.
The only exception is \SNOPT applied to the NLP-based mixing model.

As expected, the instantaneous control approach is solved very fast
for all solvers.
The single iterations are all solved in less then a second on average.
The only exception is \Ipopt applied to the
NLP-based mixing model and the implicit Euler discretization, where
some convergence issues occur within the instantaneous control
approach.
The running times required to compute the stationary solution that we
use as the initial physical state are in the same orders of magnitude
as a single instantaneous control approach iteration but slightly
longer, since no good initial point can be used by the NLP solvers.

Finally, let us also discuss the (local) optimal solutions obtained by
the different NLP solvers applied to the different model variants.
The objective function of the overall NLP consists of two parts: the
original control costs and the scaled penalty terms.
Scaling the penalty terms is always an issue in practical physical
applications for which different penalty terms have
different physical units.
Obviously, the applicability of the obtained depot control strongly
depends on the size of the penalty part of the objective, since large
slack values correspond to violated physical or technical constraints.
The table shows that different solvers find very different local
optima of the problem.
For instance, \CONOPTfour is a rather fast solver but the obtained
local optima also contain large slack values.
Contrarily, \KNITRO applied to the discretization based on central
differences computes local optima with almost vanishing slack values.
Compromising between the difference of the values in the last two columns
(which is the size of the scaled penalty terms in the objective) and the
solution times, \KNITRO applied to the discretization based on
central differences and the NLP-based mixing model seems to be the best
combination of model variant and NLP solver.

\subsection{Optimized Depot Controls}
\label{sec:optim-depot-contr}

We now present some exemplary optimal depot controls.
In \Figref{fig:control-profile-aroma} (top), the control
profile is given for the \AROMA network and the \enquote{winner
  setting} discussed in the last section.
For the given profiles, we first assume that the amount of power generated
by waste incineration is unbounded.
This leads to a control (solid line) that mainly follows the
aggregated consumption of the households (dashed line).
Due to the heat losses in the transport network, the generated power
at the depot is slightly larger than the aggregated consumption.
Since pressure losses are small in the network, the pressure increase at
the depot is almost negligible.
\begin{figure}[ht]
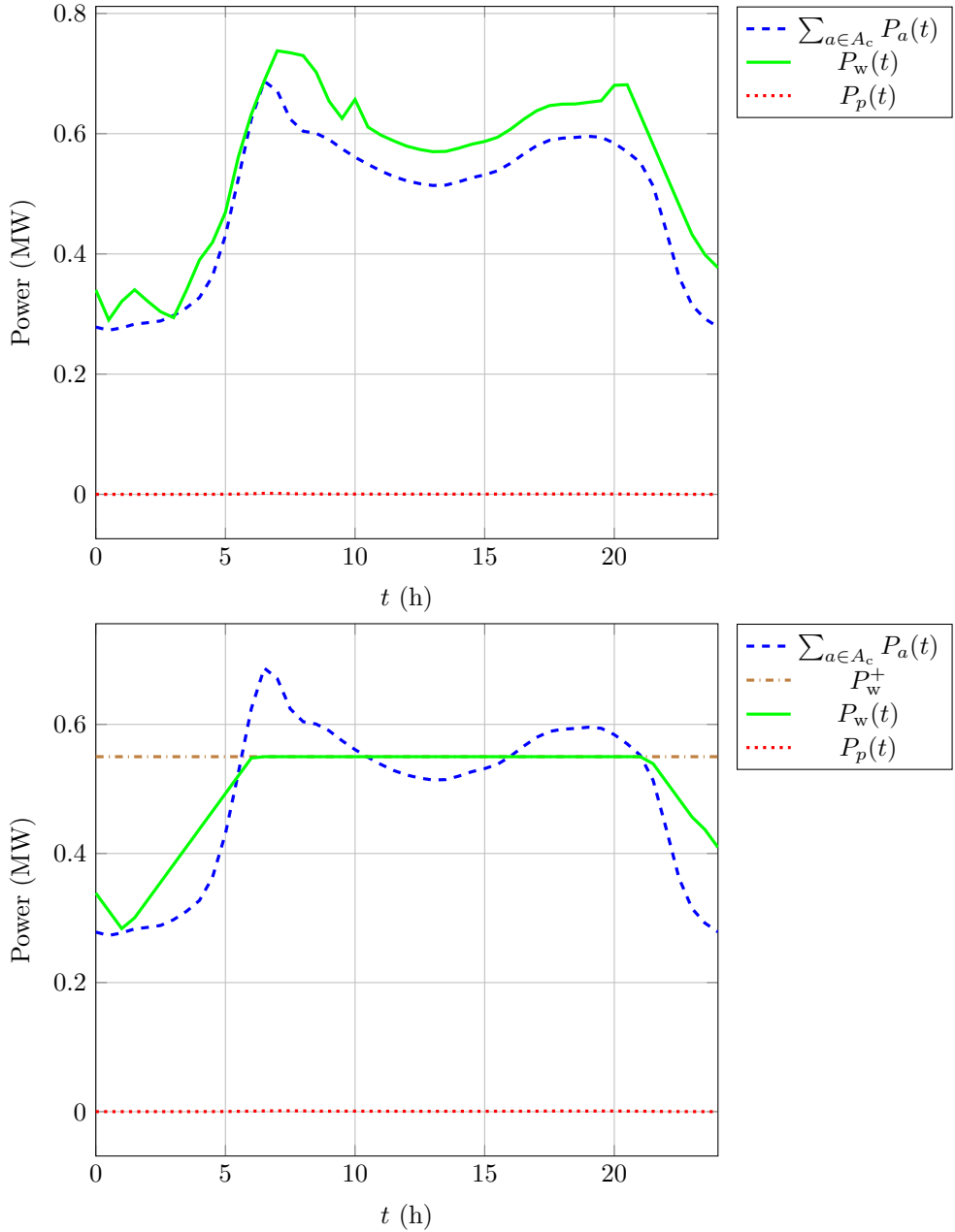

  \powerprofilenobound{Aroma_knitro_NLP_central_diff_1800_150_only_waste}
  \powerprofile{Aroma_knitro_NLP_central_diff_1800_150}
  \vspace*{-7mm}
  \caption{Aggregated power consumption (dashed curve), power generated by
    waste incineration at the depot (solid curve), and pressure
    increase at the depot (dotted curve) for the \AROMA network
    without (top) and with (bottom) waste incineration bound}
  \label{fig:control-profile-aroma}
\end{figure}
The power control qualitatively changes if power generated by waste
incineration is bounded; see the dashed-dotted line in
\Figref{fig:control-profile-aroma} (bottom).
Since aggregated power consumption is above this bound in some morning
and evening hours, the optimized power control anticipates this and
pre-heats the network in the hours before.
This is obviously required because again simply following the
aggregated consumption curve would result in hours where the power consumption
would need to be curtailed.
The same effect can be observed for the optimized depot control for
the \STREET network in \Figref{fig:control-profile-street}.
\begin{figure}[ht]
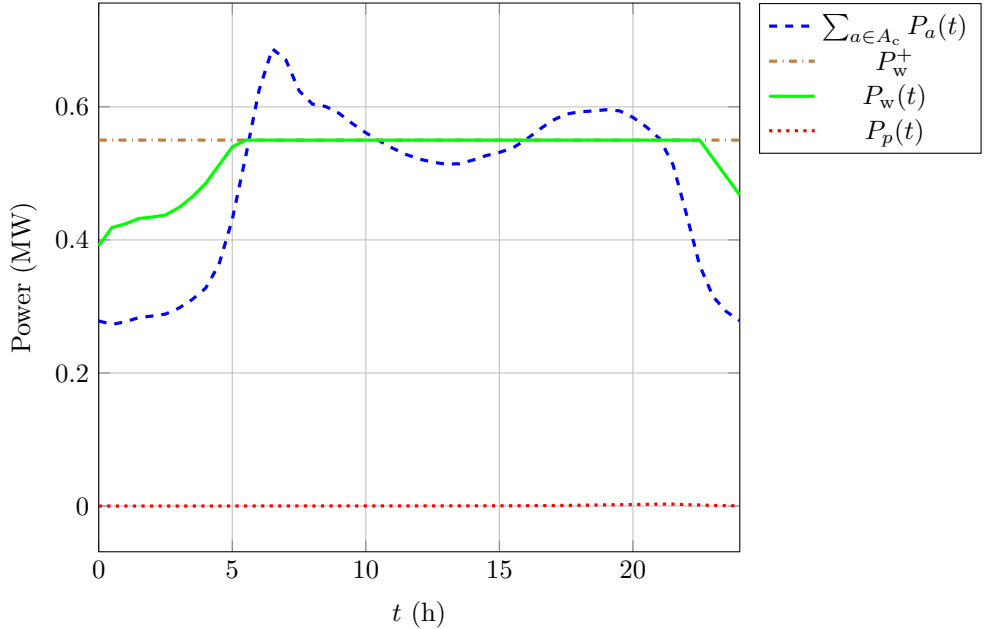

  \powerprofile{street2_ipopt_NLP_central_diff_1800_100}
  \vspace*{-7mm}
  \caption{Aggregated power consumption (dashed curve), power generated by
    waste incineration at the depot (solid curve), and pressure
    increase at the depot (dotted curve) for the \STREET network
    with waste incineration bound (dashed-dotted line)}
  \label{fig:control-profile-street}
\end{figure}
For the \STREET network, our preliminary numerical experiments
revealed that the NLP solver \Ipopt applied to the NLP-based mixing
model, the discretization scheme based on central differences as well
as $\timediff = \SI{1800}{\second}$ and $\spacediffarc =
\SI{100}{\meter}$ delivers the best results; cf. also the respective
discussion for the \AROMA network in \Secref{sec:comp-model-vari}.

Let us now finally discuss the interplay between mass flow and
water temperature on the example of the \STREET network.
Considering the power constraints of the consumers and the depot
\eqref{eq:distr-heat-consumer-power-usage} and
\eqref{eq:distr-heat-depot-power-production}, we see that power
consumption is mainly satisfied by the product of mass flow and
temperature differences.
Thus, to satisfy demand we can either increase the mass flow or the
outlet temperature of the depot.
These two values are shown in \Figref{fig:street-massflow-temp} for the
entire time horizon.
It can be seen that power consumption during night is mainly covered by
high outlet temperatures at the depot.
Here, this temperature is at its upper bound (\SI{403.15}{\K}), which is
obtained by waste incineration at the depot.
Around 4~AM it is anticipated that in the morning hours high
outlet temperatures will not be enough either due to the upper bound
of the temperature or the upper bound on waste incineration.
Thus, mass flows need to be increased, which then leads to outlet
temperatures that can be decreased.
During the remainder of the day it can be seen that mass flows and
temperatures change in an opposed way---decreasing outlet temperatures
require increased mass flows and vice versa.
\begin{figure}[ht]
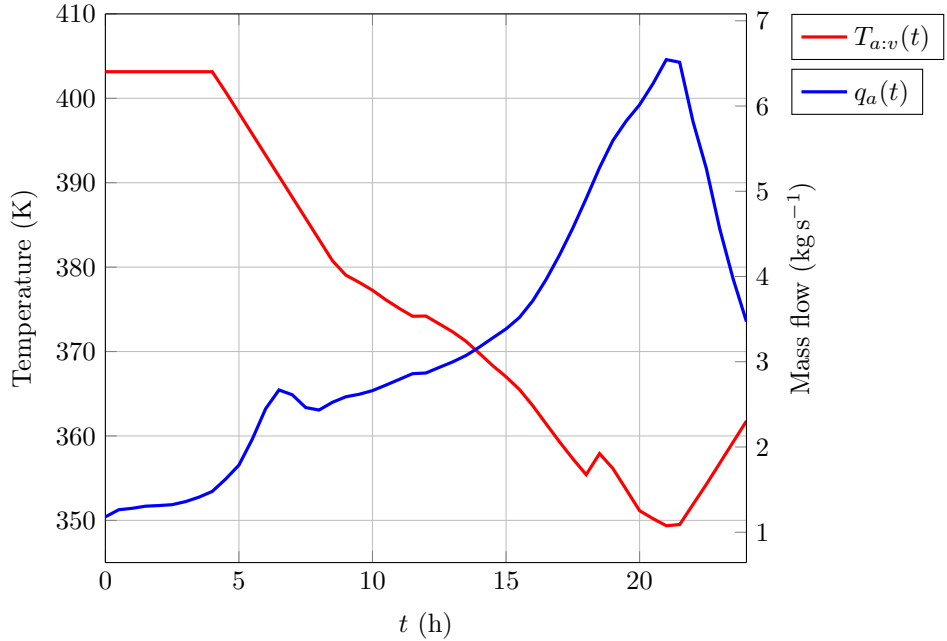

  \tempmflowprofile{street2_ipopt_NLP_central_diff_1800_100}
  \vspace*{-3mm}
  \caption{Outlet temperature and mass flow at the depot arc for the
    \STREET network.}
  \label{fig:street-massflow-temp}
\end{figure}


\section{Conclusion}
\label{sec:conclusion}

In this paper, we presented an accurate dynamic optimization
model for the control of district heating networks.
The model is mainly governed by the nonlinear partial differential
equations for water and heat flow as well as by nodal mixing models
for tracking different water temperatures in the network.
This results in a PDE-constrained MPCC or NLP model, depending on
the chosen option for the genuinely nonsmooth mixing models.
After applying suitable discretizations for the PDEs, we obtain a
finite-dimensional but large and highly nonlinear MPCC or NLP, for
which we develop different optimization techniques that then allow us
to solve realistic instances.
The applicability of the discussed models and techniques is
illustrated by a numerical case study on different networks.

The literature on mathematical optimization for district heating
networks is not as mature as for other utility networks like gas or
water networks.
Thus, many research topics remain to be addressed.
In our future work, we plan to consider adaptive techniques as
in~\cite{Mehrmann_Schmidt_Stolwijk:2017} that are based on model
hierarchies for the physics model.
Here, port-Hamiltonian modeling frameworks seem to be favorable.
A first step in this direction is already done
in~\cite{Hauschild_et_al:2019}.
In terms of the application, we think that the most urgent research
topics are to develop mathematical optimization techniques for dealing
with uncertainties (especially w.r.t. the consumption of the
households) as well as the coupling of district heating networks with
power networks.


\section*{Acknowledgments}
\label{sec:acknowledgements}

We thank \rev{the} Deutsche Forschungsgemeinschaft for their
support within projects A05, B03, and B08 in the
Sonderforschungsbereich/Transregio 154 \enquote{Mathematical
  Modelling, Simulation and Optimization using the Example of Gas
  Networks} and acknowledge the support by the German
Bundesministerium für Bildung und Forschung within
the project \enquote{EiFer}.
Moreover, we are very grateful to all the colleagues within the EiFer
consortium for many fruitful discussions on the topics of this paper
and for providing the data.


\printbibliography

\end{document}